\documentclass{amsart}

\usepackage[cmtip,all]{xy}
\usepackage{amsmath, amsthm, latexsym, amssymb, amsfonts}
\usepackage{alltt}
\usepackage[pdftex]{graphicx,color} 
\usepackage{algorithm}
\usepackage{algpseudocode}
\usepackage{longtable}
\usepackage[table]{xcolor}
\usepackage{tikz}
\usepackage{tikzscale}
\usepackage{pgfplots}
\usetikzlibrary{shapes}
\usepackage{mathtools}

\newcommand{\Char}{\operatorname{char}}

\newcommand{\dis}{\displaystyle}

\newcommand{\Hom}{\operatorname{Hom}}

\newcommand{\Ker}{\operatorname{Ker}} 

\newcommand{\h}{\operatorname{height}}

\newcommand{\reg}{\operatorname{reg}}
\newcommand{\Cl}{\operatorname{Cl}} 

\newcommand{\rank}{\operatorname{rank}}

\newcommand{\sig}{\sigma}
\newcommand{\Sig}{\Sigma}

\newcommand{\Q}{\mathbb Q}
\newcommand{\Z}{\mathbb Z}
\newcommand{\N}{\mathbb N}
\newcommand{\R}{\mathbb R}
\newcommand{\C}{\mathbb C}
\newcommand{\F}{\mathbb F}
\newcommand{\Pp}{\mathbb P}
\newcommand{\K}{{\mathbb K}}
\newcommand{\T}{{ (\K^*)}}

\newcommand{\cC}{\cl C}

\newcommand{\cl}[1]{\mathcal{#1}}

\newcommand{\la}{\langle}
\newcommand{\ra}{\rangle}
\def\aa{{\bf \alpha}}
\def\bb{\beta}
\def\a{{\bf a}}
\def\t{{\bf t}}
\def\q{{\bf q}}
\def\y{{\bf y}}

\def\uu{{\bf u}}
\def\vv{{\bf v}}
\def\x{{\bf x}}
\def\m{{\bf m}}

\def\kay{{\chi}}
\def\b{{\bf b}}
\def\e{{\bf e}}
\def\ev{{\text{ev}}}

\newcommand{\fp}{{\mathfrak p}}

\newtheorem{tm}{Theorem}[section]
\newtheorem{lm}[tm]{Lemma}
\newtheorem{coro}[tm]{Corollary}

\newtheorem{pro}[tm]{Proposition}
\newtheorem{defi}[tm]{Definition}

\newtheorem{ex}[tm]{Example}
\newtheorem{rema}[tm]{Remark}

\numberwithin{equation}{section}

\begin{document}

\title{Rational points of lattice ideals on a toric variety and toric codes}


%    author one information
\author{Mesut \c{S}ah\.{i}n}
\address{ Department of Mathematics,
  Hacettepe  University,
  Ankara, TURKEY}
\curraddr{}
\email{mesut.sahin@hacettepe.edu.tr}
\thanks{The author is supported by T\"{U}B\.{I}TAK Project No:119F177}

%    author two information
%\author{}
%\address{}
%\curraddr{}
%\email{}
%\thanks{}

%    \subjclass is required.

\date{}

\dedicatory{}

%    Abstract is required.
\begin{abstract}
We show that the number of rational points of a subgroup inside a toric variety over a finite field defined by a homogeneous lattice ideal can be computed via Smith normal form of the matrix whose columns constitute a basis of the lattice. This generalizes and yields a concise toric geometric proof of the same fact proven purely algebraically by Lopez and Villarreal for the case of a projective space and a standard homogeneous lattice ideal of dimension one. We also prove a Nullstellensatz type theorem over a finite field establishing a one to one correspondence between subgroups of the dense split torus and certain homogeneous lattice ideals. As application, we compute the main parameters of generalized toric codes on subgroups of the torus of Hirzebruch surfaces, generalizing the existing literature.
\end{abstract}

\maketitle

\section{Introduction} 
Let $\Sigma$ be a complete simplicial fan with rays generated by the lattice vectors $\vv_1,\dots,\vv_r \in \Z^n$. Each cone $\sigma \in \Sigma$ defines an affine toric variety $U_{\sigma}$ over a field $\K$ whose ring of regular functions is the semigroup ring $\K[ \check{\sigma}\cap \Z^n]$. Gluing these affine pieces in a standard way, we obtain the split toric variety $X_{\Sigma}$ as an abstract variety over $\K$. The Cox ring $S=\K[x_1,\dots,x_r]$ of $X_{\Sigma}$ is graded naturally by the group $\Cl X_{\Sigma}$, i.e. we have the decomposition
$S=\displaystyle \bigoplus_{\aa \in \Cl X_{\Sigma}} S_{\aa}$, where $S_{\aa}$ is the $\K$-vector space spanned by the monomials of degree $\aa$.

By a celebrated result due to Cox in \cite{CoxRingToric}, $\overline{\K}$-rational points of a toric variety $X:=X_{\Sigma}$ over an algebraic closure $\overline{\K}$ of $\K$, is identified with the orbits in the Geometric Invariant Theory quotient $[\overline{\K}^r\setminus V(B)] /G$, see Section \ref{s:prelim} for details. It follows that the $\K$-rational points $X(\K)$ of $X$ can have representatives with coordinates in $\K$. Thus, we can study $\K$-rational points of a subvariety of $X$ cut out by a $\Cl X_{\Sigma}$-graded (or simply homogeneous) ideal $J$ using this correspondence:
\[
V_{X(\K)}(J)=\{ [P]\in [\K^r\setminus V(B)] /G  \: : \: F(P)=0, \text{ for any homogeneous } F\in J \}.
\]

Our main motivation comes from immediate applications to error-correcting codes defined on toric varieties over a finite field $\K=\F_q$. In this regard, we assume that the class group $\Cl X_{\Sigma}$ is free throughout, which does not harm the generality, as the codes on a singular toric variety survive when the singularity is resolved and smooth complete toric varieties have a torsion-free class group, by \cite[Proposition 4.2.5]{CLS-ToricVarieties}.

 For any $\aa$ and any subset $Y=\{[P_1],\dots,[P_N]\}$ consisting of some $\F_q$-rational points of $X$, we have the following {\it evaluation map} for $\K=\F_q$:
\begin{equation}
\dis \ev_{Y}:S_\aa\to \K^N,\quad F\mapsto \left(F(P_1),\dots,F(P_N)\right).
\end{equation}
The image $\cl{C}_{\aa,Y}=\text{ev}_{Y}(S_\aa)$ is a linear code, called a {\it generalized toric code}. The three basic parameters of $\cl C_{\aa,Y}$ are block-length which is $N$, the dimension which is $K=\dim_{\F_q}(\cl C_{\aa,Y})$, and the minimum distance $\delta=\delta(\cl{C}_{\aa,Y})$ which is the minimum of the numbers of nonzero components of nonzero vectors in $\cl C_{\aa,Y}$.

When $Y=T_X(\F_q)$ is the $\F_q$-rational points of the maximal torus of $X$ we get the classical toric code examined for the first time by Hansen in \cite{HansenAAECC} and studied further in e.g. \cite{JoynerAAECC,DR2009,OrderBound,ToricCIcodesIS13,LittleFFA2013,7championCodes} showcasing some champion codes. Although the block-length of $\cC_{\aa,Y}$ equals $N=(q-1)^n$ in this classical case, its computation depends on the way $Y$ is presented in the general case. For instance, if $Y=Y_Q$ is parameterised by the monomials whose power is a column of an integer matrix $Q$, then some algorithms are given for computing the length $N=|Y|$, see \cite{AlgebraicMethods2011,BaranSahin}. The vanishing ideal $I(Y)\subset S$ of $Y$, which is generated by homogeneous polynomials vanishing on $Y$, is relevant in the computation of the other parameters as well. Namely, the kernel of $\ev_{Y}$ is just the subspace $I_{\aa}(Y):=I(Y)\cap S_{\aa}$, and hence the dimension $K=\dim_{\F_q}(\cl C_{\aa,Y})$ is the value $H_Y(\aa):=\dim_{\F_q} S_{\aa}-\dim_{\F_q}I_{\aa}(Y)$ of the multigraded Hilbert function $H_Y$ of $Y$, see \cite{MultHFuncToricCodes16}. The minimum distance function of $I(Y)$ is introduced and is related to the minimum distance of $\cC_{\aa,Y}$ in \cite{JBYPRV2017} for the case of the projective space $X=\Pp^n$.  The ideal $I(Y_Q)$  is a lattice ideal $I_L$ for a unique lattice $L\subseteq \Z^r$ of rank $n$ as shown in \cite{AlgebraicMethods2011,CodesWPS,BaranSahin}. Moreover, every subgroup of $T_X(\F_q)$ is not only of the form $Y_Q$ but also is the common zeroes $V_{X(\F_q)}(I_L)$ of binomials generating a lattice ideal $I_L$, see \cite[Section 3]{sahin18}.

The present paper aims mainly to develop methods for computing parameters of generalized toric codes $\cl C_{\aa,Y}$ from subgroups $Y=V_{X(\F_q)}(I_L)$ of $T_X$. Before we state our main results, we need to introduce a little more notation. 

Let $\cl B$ be an $r\times n$ matrix whose columns constitute a basis for $L$; let $\cl A$ and $\cl C$ be unimodular matrices of sizes $r$ and $n$, respectively, so that $$\cl D=\cl A \cl B \cl C=[d_1\e_1 |\cdots |d_n\e_n]$$ is the Smith-Normal form of $\cl B$, where $\e_1,\dots,\e_n \in \Z^r$. These $d_i$'s are known as the \textit{invariant factors} having the property that $d_i$ divides $d_{i+1}$ for each $i\in [n-1]:=\{1,\dots,n-1\}$ and that the order of the torsion part $T(\Z^r /L)$ of the group $\Z^r /L$ is $d_1\cdots d_n$, by the fundamental structure theorem for finitely generated abelian groups, see \cite[pp. 187–-188]{Jacobson}.

 We first observe in Proposition \ref{p:VX(IL)lyingInTX} that the common solutions $V_{X(\K)}(I_L)$ of polynomials in $I_L$ lie in the torus $T_X$, for any field $\K$. 
 
 Then, we count characters to see the number of $\K$-rational points is finite: $|V_{X(\K)}(I_L)| \leq d_1\cdots d_n.$ Furthermore, it meets this upper bound if and only if $\K$ has all the $d_i$-th roots of unity, for every $i \in [n]$, see Theorem \ref{t:K-rationalpoints}.  
 
 Let $Y$ be a subgroup of $T_X(\F_q)$. Then, $I(Y)=I_L$ for a unique lattice $L$ with $(q-1)L_{\beta}\subseteq  L\subseteq L_{\beta}$, see Lemma \ref{l:latticeOFsubgroup}. Moreover, the number of $\F_q$-rational points in $Y$ is given in Theorem \ref{t:rationalpoints} by $$|Y|=|V_{X(\F_q)}(I_L)|=|L_{\beta}/L|=d_1\cdots d_n.$$ Therefore, the length of the code $\cl C_{\aa}(Y)$ will be $d_1\cdots d_n$. 
 
Furthermore, we observe that the set $V_{X(\F_q)}(I_L)$ of $\F_q$-rational points coincide with the set of $\overline{\F}_q$-rational points $V_X(I_L):=V_{X(\overline{\F}_q)}(I_L)$ in Lemma \ref{l:Subgroups=overClosure} for lattices with $(q-1)L_{\beta}\subseteq  L\subseteq L_{\beta}$. As a consequence, we also prove a Nullstellensatz type result over the finite field $\F_q$ saying that $I(V_X(I_L))=I_L$ establishing a $1-1$ correspondence between the subgroups $V_X(I_L)$ and the lattice ideals $I_L$, where $(q-1)L_{\beta}\subseteq  L\subseteq L_{\beta}$, see Theorem \ref{t:nullstellensatz}. This result will eliminate the need for an algorithm to compute the vanishing ideal of the subgroup $V_X(I_L)$, which will save time in search of codes with good parameters. In literature, Nullstellensatz type results over finite fields have been obtained (see \cite{Ghorpade19} for a nice survey) and used effectively for computing the basic parameters of evaluation codes on subvarieties of affine and projective spaces, see \cite[Theorems 3.12, 3.13]{Jaramillo21DCC} and \cite[Corollary 4.4]{AlgebraicMethods2011}.  

If $I_L$ is a homogeneous ideal of dimension $d=r-n$, then its degree is just the constant Hilbert polynomial, which is computed to be $\deg(I_L)=|L_{\bb}/L|=d_1\cdots d_n$, in Theorem \ref{t:DegreeOfI_L}, generalizing the nice result  \cite[Theorem 3.13]{HLRV2014}. This will be nothing but the mixed volume of the Newton polytopes of binomials generating $I_L$ by Corollary \ref{c:mixedvolume}. It is worth to mention \cite[Theorem 4.6]{O'Carroll14} also generalizing \cite[Theorem 3.13]{HLRV2014} and giving a nice formula for the degree of $I_L$, when $X$ is an affine space and $L$ is an arbitrary lattice, see also \cite[Theorem 9.4.2]{VillarrealBOOK}.

Finally, we uncover the structure and give two descriptions for the points of the subgroup cut out by the lattice ideal $I_L$. More precisely, we have
\[\dis V_{X(\K)}(I_L)=\la [P_1] \ra \times \dots \times \la [P_n] \ra,
\]
for the point $[P_i]=[\eta_i^{a_{i1}}:\cdots:  \eta_i^{a_{ir}}]$, where $\eta_i$ is a primitive $d_i$-th root of unity in $\K$, for each $i\in [n]$. In addition, $$V_{X(\F_q)}(I_L)=Y_Q=\{[{\bf t}^{\q_{1}}:\cdots:{\bf t}^{\q_{r}}]|{\bf t}\in (\F_q^{*})^s\}$$ is parameterized by the columns of the matrix $Q=[\q_1 \q_2\cdots \q_r]$, whose $i$-th row is  $(q-1)/d_i$ times the $i$-th row of the matrix $\cl A$, see Theorem \ref{t:parameterization}. 

As application, we compute the main parameters of generalized toric codes on subgroups of the torus of Hirzebruch surfaces, generalizing the existing literature, see Theorem \ref{T:codesOnHirzebruch}. The parameterization above is especially useful in computing the minimum distance.

\section{preliminaries} \label{s:prelim}

In this section we fix our notation and recall standard facts from toric geometry relying on the most recent wonderful book \cite{CLS-ToricVarieties}, although there are other well known excellent books.

Let $\Sigma$ be a complete simplicial fan with rays generated by the lattice vectors $\vv_1,\dots,\vv_r \in \R^n$. The fan gives the following short exact sequence: 
\begin{equation}\label{eq:SES}
\dis \xymatrix{ \mathfrak{P}: 0  \ar[r] & \Z^n \ar[r]^{\phi} & \Z^r \ar[r]^{{\bb}} & \Z^d \ar[r]& 0},
\end{equation}
where $\phi$  denotes  the matrix  $[\vv_1\cdots \vv_r]^T$ and $d=r-n$ is the rank of the class group $\Cl X$ of $X:=X_{\Sigma}$. Denote by $L_{\bb}$ the key sublattice of $\Z^r$ which is isomorphic to $\Z^n$ via $\phi$, whose basis is given by the columns $\uu_1,\dots, \uu_n$ of $\phi$.

The homogeneous coordinate ring $\overline{\K}[x_1,\dots,x_r]$ of the toric variety $X$ over $\overline{\K}$ is graded naturally by the columns of the matrix $\bb$ via $\deg(x_j)=\bb_j$, for each $j\in [r]$. A polynomial whose monomials $x_1^{a_1}\cdots x_r^{a_r}$ have the same degree $\aa=a_1\bb_1+\cdots+a_r\bb_r$ is called $\bb$-graded, $L_{\bb}$-homogeneous or just homogeneous when there is no harm of confusion. We are interested in sublattices $L$ of $L_{\bb}$ and the $L_{\bb}$-homogeneous \textit{lattice ideals} they define in the Cox ring $\K[x_1,\dots,x_r]$ of $X(\K)$: 
$$I_L=\la \x^{\m^+}- \x^{\m^-} \: | \: \m=\m^{+}-\m^{-}\in L \ra. $$

Applying $\Hom(-,\overline{\K}^*)$ functor to $\mathfrak{P}$ we get the following dual short exact sequence since $\overline{\K}^*$ is divisible:
\begin{equation}\label{eq:dualSES}
\dis \xymatrix{ \mathfrak{P}^*: 1  \ar[r] & G \ar[r]^{i} & (\overline{\K}^*)^r \ar[r]^{\pi} & (\overline{\K}^*)^n \ar[r]& 1},
\end{equation}
where $\pi:P \mapsto (\x^{\uu_1}(P), \dots , \x^{\uu_n}(P))$ is expressed using the notation 
\[\x^{\a}(P)=p_1^{a_{1}}\cdots p_r^{a_{r}} \mbox{ for } P=(p_1,\dots,p_r)\in (\K^*)^r \mbox{ and for } \a=(a_1,\dots,a_r)\in \Z^r.\]

 Notice also that the algebraic group $G=\ker(\pi)$ is the zero locus in $(\overline{\K}^*)^r$ of the toric (prime lattice) ideal $I_{L_{\bb}}$:
\begin{align*}
V(I_{L_{\bb}})\cap (\overline{\K}^*)^r:&=\{P \in (\overline{\K}^*)^r \: | \: (\x^{\m^+}- \x^{\m^-})(P)=0 \:\mbox{for all} \: \m\in L_{\bb} \}\\
&=\{P \in (\overline{\K}^*)^r \: | \: \x^{\m}(P)=1 \:\mbox{for all} \: \m\in L_{\bb} \},    
\end{align*}
where $V(J)$ denotes the usual affine subvariety of $\overline{\K}^r$ cut out by the ideal $J$.

As proved by Cox in \cite{CoxRingToric}, when $\overline{\K}$ is an algebraic closure of $\K$, the toric variety $X$ over $\overline{\K}$, is isomorphic to the geometric quotient $[\overline{\K}^r\setminus V(B)] /G$, where $B$ is the monomial ideal in $\K[x_1,\dots,x_r]$ generated by the monomials 
\[
\dis \x^{\hat{\sig}}=\Pi_{\rho_i \notin \sig}^{} \: x_i \text{ corresponding to cones } \sig \in \Sig.
\]
Hence, $\K$-rational points $X(\K)$ are orbits $[P]:=G\cdot P$, for $P\in \K^r\setminus V(B)$. By (\ref{eq:dualSES}), we have $T_X=(\overline{\K}^*)^r/G$ and since the toric variety is split we have $T_X(\K)\cong \T^n$. By resolving the singularity, if needed, we may assume without loss of generality that there is at least one smooth cone $\sigma \in \Sigma$. Hence, the first $n$ rows of the matrix $\phi$ in (\ref{eq:SES}) are $\vv_1=\e_1, \dots, \vv_n=\e_n$. Thus, the map $\pi$ in (\ref{eq:dualSES}) becomes 
$$\pi(x_1,\dots,x_r)=(x_1 \x^{\uu_1-\e_1},\dots,x_n \x^{\uu_n-\e_n}).$$ Hence, given $(t_1,\dots,t_n)\in (\overline{\K}^*)^n$, solving $x_i\x^{\uu_n-\e_n}=t_i$, we get 
\[x_i=t_i\x^{\uu_i-\e_i}=t_ix_{n+1}^{-u_{n+1,i}}\cdots x_{r}^{-u_{r,i}}, \text{ for all } i\in [n].
\]
As a result, the set $\pi^{-1}(t_1,\dots,t_n)$ consists of the points 
\begin{eqnarray}\label{eq:piInverse}
 (t_1\lambda_1^{-u_{n+1,1}}\cdots \lambda_d^{-u_{r,1}},\dots,t_n\lambda_1^{-u_{n+1,n}}\cdots \lambda_d^{-u_{r,n}},\lambda_1,\dots,\lambda_d)   
\end{eqnarray}
for $(\lambda_1,\dots,\lambda_d)=(x_{n+1},\dots,x_r)\in (\overline{\K}^*)^d$, yielding the parameterization:
\[
G=\{(\lambda_1^{-u_{n+1,1}}\cdots \lambda_d^{-u_{r,1}},\dots,\lambda_1^{-u_{n+1,n}}\cdots \lambda_d^{-u_{r,n}},\lambda_1,\dots,\lambda_d): \lambda_i \in \overline{\K}^*\}.
\]

When $(t_1,\dots,t_n)\in (\K^*)^n$, there is a point $P=(t_1,\dots,t_n,1\dots,1)\in (\K^*)^r$ with $\pi(P)=(t_1,\dots,t_n)$. Therefore, points of $T_X(\K)=\T^r/G$ have representatives with components from $\K^*$. Furthermore, equivalence of two points $(x_1,\dots,x_r),(x'_1,\dots,x'_r)\in \T^r$ with respect to the action of $G$ and that of $G(\K)$ are the same. This is because, when 
$G\cdot (x_1,\dots,x_r)=G\cdot (x'_1,\dots,x'_r)$, there is a point $(g_1,\dots,g_r)\in G \subset (\overline{\K}^*)^r$ with $x_i=g_ix'_i$ for all $i\in [r]$. In particular, we have $x_{n+i}=\lambda_ix'_{n+i}$ forcing $\lambda_i=x_{n+i}/x'_{n+i} \in \K$, for all $i\in [d]$.

For an $L_{\bb}$-graded ideal $J$, we introduce the subvariety of the toric variety $X$ defined by $J$ as follows
\[
V_{X}(J)=\{ [P]\in X(\overline{\K})  \: : \: F(P)=0, \text{ for any homogeneous } F\in J \}.
\] Its $\K$-rational points form the subset
\[
V_{X(\K)}(J)=\{ [P]\in X(\K)  \: : \: F(P)=0, \text{ for any homogeneous } F\in J \}.
\]
Therefore, any point $[P]\in V_{X(\K)}(J)$ is an equivalence class $G\cdot P$ of a point $P$ from $V(J)\setminus V(B) \subseteq \K^r\setminus V(B)$, for the subgroup $G\subset (\overline{\K}^*)^r$. 

As the elements of the algebraic torus $(\overline{\K}^*)^r$ in $\mathfrak P^* $ are already identified with the characters $\rho: \Z^r \rightarrow \overline{\K}^*$ and $T_X(\K)=\T^r/G$, we take the approach to investigate characters $\rho: \Z^r \rightarrow \K^*$ in order to compute the number of $\K$-rational points of the lattice ideal $I_L$ on the toric variety $X$.

\section{Subgroups defined by lattice ideals over any field}
\label{sec:subgroupsK}

In this section, we study subgroups of the torus $T_X$ defined by lattice ideals $I_L$ inside $\K[x_1,\dots,x_r]$, which is an $L_{\bb}$-graded ring via $\deg(x_i)=\bb_i$, where $\bb_i$ is the $i$-th column of the matrix $\bb$ in (\ref{eq:SES}) and $L$ is a sublattice of $L_{\bb}$.

First of all, $\rank L_{\bb}=n$ since $L_{\bb}$ is isomorphic to $\Z^n$ via the map $\phi$ in the sequence $\mathfrak{P}$ in (\ref{eq:SES}). On the other hand, it is well-known that $\dim J= r-\h (J)$, where $\dim J$ is the Krull dimension of the quotient ring $\K[x_1,\dots,x_r]/J$. So, we have
$$\rank L=n=\h(I_L) \iff \dim I_L= d=r-n.$$

As the vanishing ideals of subgroups of the torus $T_X$ are $L_{\bb}$-homogeneous lattice ideals of Krull-dimension $d$ by \cite[Proposition 2.6]{sahin18}, and $I_L$ is $L_{\bb}$-homogeneous if and only if $L\subseteq L_{\bb}$ by \cite[Proposition 2.3]{sahin18}, we focus on lattices $L$ satisfying $L\subseteq L_{\bb}$ and $\rank L=n$. 

\begin{lm}\label{l:(ell*satL)SubsetL} Let $L\subseteq L_{\bb}$ and $\rank L=n$. Then, $\ell L_{\bb} \subseteq L$ for the order $\ell=|L_{\bb}/L|$.
\end{lm}
\begin{proof}
 Since $L_{\bb}/L$ is a finite group of order $\ell$, we have $\ell(\m+L)=L$, for every $\m\in L_{\bb}$. This means that $\ell \m \in L$ for all $\m\in L_{\bb}$.
\end{proof}

\begin{lm}\label{l:latticeComparison} $L\subseteq L'$ if and only if $I_L\subseteq I_{L'}$.
\end{lm}
\begin{proof}($\Rightarrow$:) If $L\subseteq L'$, then every $\m\in L$ lies in $L'$. Therefore, every binomial generator $\x^{\m^+}-\x^{\m^-}$ of $I_L$ lies in $I_{L'}$. So, $I_L\subseteq I_{L'}$. \\
($\Leftarrow$:) Assume that $I_L\subseteq I_{L'}$ and $\m\in L$. Then, the binomial $\x^{\m^+}-\x^{\m^-}$ of $I_L$ lies in $I_{L'}$. It follows from \cite[Lemma 3.2]{HLRV2014} that $\m \in L'$.
\end{proof}

\begin{pro}\label{p:VX(IL)lyingInTX} Let $L\subseteq L_{\bb}$ and $\rank L=n$. Then, $V_{X(\K)}(I_L)$ is a subgroup of $T_X(\K)$.
\end{pro}
\begin{proof} By Lemma \ref{l:(ell*satL)SubsetL}, we have $\ell L_{\bb}\subseteq L \subseteq L_{\bb}$ for the order $\ell=|L_{\bb}/L|$. As smaller lattices correspond to smaller lattice ideals by Lemma \ref{l:latticeComparison}, it follows that $I_{\ell L_{\bb}}\subseteq I_L \subseteq I_{L_{\bb}}$. Thus, we have the inclusions below: 
\[
\{[1]\}=V_{X(\K)}(I_{L_{\bb}})\subseteq V_{X(\K)}(I_L)\subseteq V_{X(\K)}(I_{\ell L_{\bb}}),
\]
where $[1]$ denotes the orbit of the point $(1,\dots,1)\in \K^r \setminus V(B)$. So as to prove that $V_{X(\K)}(I_L)$ is a subset of the torus $T_X(\K)$, it suffices to reveal that $V_{X(\K)}(I_{\ell L_{\bb}})\subseteq T_X(\K)$. Given any point $[P]\in V_{X(\K)}(I_{\ell L_{\bb}})$, we have $P\in V(I_{\ell L_{\bb}})\setminus V(B)$. By definition, $P\in V(I_{\ell L_{\bb}})$ exactly when $\x^{\ell\m^{+}}(P)=\x^{\ell\m^{-}}(P)$ for all $\m=\m^{+}-\m^{-}\in L_{\bb}$. Since $\x^{\ell\m^{+}}(P)=\x^{\ell\m^{-}}(P)$ is equivalent to $\x^{\m^{+}}(P^\ell)=\x^{\m^{-}}(P^\ell)$, it follows that $P\in V(I_{\ell L_{\bb}})$ if and only if $P^\ell\in V(I_{L_{\bb}})$. As $B$ is a monomial ideal, $P\notin V(B)$ if and only if $P^\ell \notin V(B)$. Thus, $P\in V(I_{\ell L_{\bb}})\setminus V(B)$ implies that $P^\ell\in V(I_{L_{\bb}})\setminus V(B)=G$. As $G$ lies in the torus $(\overline{\K}^*)^r$, we have $P\in \T^r$ yielding $[P]\in T_X(\K)$. Hence, $V_{X(\K)}(I_{L})\subseteq T_X(\K)$. 

If $[P],[Q]\in V_{X(\K)}(I_{L})$, then we have the following:
$$\x^{\m^{+}}(P\cdot Q^{-1})=\x^{\m^{+}}(P)\x^{\m^{+}}( Q)^{-1}=\x^{\m^{-}}(P)\x^{\m^{-}}( Q)^{-1}=\x^{\m^{-}}(P\cdot Q^{-1}).$$
Thus, $[P]\cdot [Q]^{-1}$ lies in $V_{X(\K)}(I_{L})$ demonstrating that it is indeed a subgroup.
\end{proof}

\begin{rema} It is easy to see that $V_X(I_L)\cap T_X$ is a subgroup of $T_X$, when we only have $L\subseteq L_{\bb}$, see also \cite[Corollary 2.4]{sahin18}. But, if $\rank L < n$, then $V_X(I_L)$ need not lie inside the torus $T_X$, see \cite[Example 3.6]{sahin18}.
\end{rema}

 For any $\m=\m^+-\m^-\in L$, the binomial $\x^{\m^+}-\x^{\m^-}\in I_L$, hence we have $\x^{\m^+}(P)=\x^{\m^-}(P)$ for a point $[P]\in V_{X(\K)}(I_L)$. If we further assume that $L\subseteq L_{\bb}$ and $\rank L=n$, then $\x^{\m}(P)=1$ as $V_{X(\K)}(I_L)\subseteq T_X(\K)$. Therefore, a point $[P]\in V_{X(\K)}(I_L)$, defines a character $\dis \kay_{P}: L_{\bb}/L  \rightarrow \K^*$ via 
 \[
 \kay_P(\m+L):=\x^{\m}(P)=p_1^{m_1}\cdots p_r^{m_r},
 \] for all $\m=(m_1,\dots,m_r)\in L$ under these circumstances.

\begin{lm}\label{l:Subgroup=characters} Let $L\subseteq L_{\bb}$ and $\rank L=n$. Then the groups $V_{X(\K)}(I_L)$ and $\Hom(L_{\bb}/L,\K^*)$ are isomorphic.
\end{lm}
\begin{proof} Let $\psi : V_{X(\K)}(I_L) \rightarrow \Hom(L_{\bb}/L,\K^*)$ be defined by $[P]\rightarrow \kay_P$. 
\begin{itemize}
    \item The map $\psi$ is well-defined, for if $g\in G=V(I_{L_{\bb}})\cap (\overline{\K}^*)^r$, then $\kay_{gP}=\kay_{P}$, since for all $\m\in L_{\bb}$, we have $$\x^{\m}(gP)=g^{\m}\x^{\m}(P)=\x^{\m}(P).$$ 
    \item $\psi$ is a homomorphism, since $\kay_{P\cdot P'}=\kay_P \cdot \kay_{P'}$. 
    \item It is injective, as $\kay_P(\m+L)=\x^\m(P)=1$, for all $\m\in L_{\bb}$, implies that $P\in G$ that is $[P]=[1]$.
    \item It is surjective, as we prove now. Let $\{\uu_1,\dots,\uu_n\}$ be the $\Z$-basis for the lattice $L_{\bb}$. For a given $\kay:L_{\bb}/L  \rightarrow \K^*$, let $t_i=\kay(\uu_i+L)$ for each $i\in [n]$. Since $\mathfrak{P}^*$ is exact and we can choose $(\lambda_1,\dots,\lambda_n)\in (\K^*)^d$ in (\ref{eq:piInverse}), there is a point $P \in (\K^*)^r$ such that $\pi(P)=(\x^{\uu_1}(P),\dots,\x^{\uu_n}(P))=(t_1,\dots,t_n)$. Every element $\m\in L_{\bb}$ is written as $\m=c_1\uu_1+\cdots+c_n\uu_n$ for some $(c_1,\dots,c_n)\in \Z^n$. Thus, we have $$\kay(\m+L)=t_1^{c_1}\cdots t_n^{c_n}=(\x^{\uu_1}(P))^{c_1} \cdots (\x^{\uu_n}(P))^{c_n}=\kay_P(\m+L).$$ 
If $\m\in L$, then it follows that $\x^{\m}(P)=\kay(\m+L)=1$ for all $\m\in L$ and so $[P]\in V_X(I_L)$. 
Hence, we have $\kay=\kay_P$.  
    \end{itemize}
These complete the proof.
\end{proof}

\begin{lm}\label{l:characters=elliffRootsOf1} Let $L\subseteq L_{\bb}$ and $\rank L=n$. If $d_i$ is an invariant factor of the group $L_{\bb}/L$ and $C_i:=\{c\in \K^* : c^{d_i}=1\}$ is the group of $d_i$-th roots of unity lying inside $\K$, for each $i\in [n]$, then the groups $\Hom(L_{\bb}/L,\K^*)$ and $C_1\times \dots \times C_n$ are isomorphic.
\end{lm}
\begin{proof} Let $\{\m_1,\dots,\m_n\}$ be a $\Z$-basis for the lattice $L_{\bb}$ such that $$\{d_1\m_1,\dots,d_n\m_n\}$$ is a $\Z$-basis of $L$. 

We regard elements in $\Hom(L_{\bb}/L,\K^*)$ as (partial) characters $\rho: L_{\bb}\rightarrow \K^*$ whose restriction to $L$ is the identity $1_L$. So, we have $\rho(\m_i)^{d_i}=\rho(d_i\m_i)=1$, for each $i\in [n]$. Since $\rho$ is determined by the values $\rho(\m_1),\dots,\rho(\m_n)$, the map
$\Theta: \Hom(L_{\bb}/L,\K^*) \rightarrow C_1\times \dots \times C_n$,  defined by $\Theta(\rho)=(\rho(\m_1),\dots,\rho(\m_n))$ is well-defined.
\begin{itemize}
    \item $\Theta$ is a homomorphism as $\Theta(\rho \cdot \rho')=\Theta(\rho) \cdot \Theta(\rho')$.
    \item It is injective, since $\Theta(\rho)=(\rho(\m_1),\dots,\rho(\m_n))=(1,\dots,1)\implies \rho=1_{L_{\bb}}$.
    \item $\Theta$ is surjective as we demonstrate below. 
    
    Every choice of $(c_1,\dots,c_n)\in C_1\times \dots \times C_n$ determines an element 
$$\rho: L_{\bb}\rightarrow \K^* \quad \mbox{ via } \quad \rho(k_1\m_1+\cdots+k_n\m_n)=c_1^{k_1}\cdots c_n^{k_n}.$$ 
Furthermore, if $\m=k_1\m_1+\cdots+k_n\m_n\in L$ then $d_i$ divides $k_i$ for all $i$ and thus $\rho(\m)=1$, as $c_1^{k_1}=\cdots =c_n^{k_n}=1$. Hence, $\rho\in \Hom(L_{\bb}/L,\K^*)$.
\end{itemize}
Thus, $\Theta$ is the required isomorphism, completing the proof. 
\end{proof}

\begin{tm}\label{t:K-rationalpoints} Let $L\subseteq L_{\bb}$ and $\rank L =n$. If $d_i$ is an invariant factor of the group $L_{\bb}/L$ and $C_i:=\{c\in \K^* : c^{d_i}=1\}$ is the group of $d_i$-th roots of unity lying inside $\K$, for each $i\in [n]$, then $V_{X(\K)}(I_L)\cong C_1\times \dots \times C_n$, and thus the number of $\K$-rational points $|V_{X(\K)}(I_L)|$ is $|C_1|\cdots|C_n| \leq d_1\cdots d_n$. Moreover, $|V_{X(\K)}(I_L)|=d_1\cdots d_n \iff \K$ has all the $d_i$-th roots of unity, for every $i \in [n]$.
\end{tm}
\begin{proof} It follows from Lemma \ref{l:Subgroup=characters} that $V_{X(\K)}(I_L)\cong \Hom(L_{\bb}/L,\K^*)$. We also have $\Hom(L_{\bb}/L,\K^*) \cong C_1\times \dots \times C_n$ by Lemma \ref{l:characters=elliffRootsOf1}, whence $V_{X(\K)}(I_L)\cong C_1\times \dots \times C_n$. Since $|C_i|\leq d_i$ in general, we have $|V_X(I_L)|=|C_1|\cdots|C_n| \leq d_1\cdots d_n$. 

The second part follows, since $|C_i|= d_i \iff $ $\K$ has all the $d_i$-th roots of unity, for every $i \in [n]$.
\end{proof}
%\begin{rema} $\Ext_{\Z}^1(\cA,\K^*)=1$ is satisfied when $\cA$ is free or $\K=\overline{\F}_q$.\end{rema}

\section{Subgroups defined by lattice ideals over a finite field}
\label{sec:subgroupsF_q}

In this section, we compute the number of $\F_q$-rational points of a subgroup of the algebraic group $T_X$, where $X$ is a toric variety over the finite field $\F_q$. By the virtue of Theorem \ref{t:K-rationalpoints}, this amounts to computing the order of the cyclic subgroup $C_i$ of $d_i$-th roots of unity lying inside the cyclic group $\F_q^*$ of order $q-1$.

\begin{tm}\label{t:Fq-rationalpoints}  If $L\subseteq L_{\beta}$ and $\rank L=n$, then the number of $\F_q$-rational points is given by $|V_{X(\F_q)}(I_L)|=(d_1,q-1)\cdots (d_n,q-1)$, where $(d_i,q-1)$ is the greatest common divisor of $q-1$ with the invariant factor $d_i$. 
\end{tm}
\begin{proof} Fix $i\in [n]$. Recall that $C_i=\{c\in \K^* : c^{d_i}=1\}$ is a cyclic subgroup of $\F_q^*$ of order $|C_i|$, and so $|C_i|$ divides $q-1$. The order $|C_i|$ of a generator of $C_i$ divides $d_i$, as well. Hence, $|C_i|$ divides $(d_i,q-1)$. On the other hand, as $(d_i,q-1)$ divides $q-1$, there is a cyclic subgroup of $\F_q^*$ of order $(d_i,q-1)$, which is clearly contained in $C_i$, since $(d_i,q-1)$ divides $d_i$. Therefore, $|C_i|=(d_i,q-1)$. 

The proof now follows directly from Theorem \ref{t:K-rationalpoints}.
\end{proof}

Before going further, let us single out lattices corresponding to subgroups of $T_X$.

\begin{lm}\label{l:latticeOFsubgroup} If $Y$ is a subgroup of $T_X(\F_q)$ then  $I(Y)=I_L$ for a unique lattice $L$ with $(q-1)L_{\beta}\subseteq  L\subseteq L_{\beta}$. Therefore, subgroups of $T_X(\F_q)$ are of the form $V_{X(\F_q)}(I_L)$ for lattices with $(q-1)L_{\beta}\subseteq  L\subseteq L_{\beta}$.
\end{lm}
\begin{proof} By \cite[Theorem 2.9]{sahin18} $I(Y)=I_L$ for a unique lattice $L$ with $ L\subseteq L_{\beta}$. Since $Y\subseteq T_X(\F_q)$ implies $I(T_X(\F_q))\subseteq I(Y)$ and we have $I(T_X(\F_q))=I_{(q-1)L_{\beta}}$ by \cite[Corollary 4.14 (ii)]{sahin18}, the first claim follows from Lemma \ref{l:latticeComparison}.

If $Y$ is a subgroup of $T_X(\F_q)$, then $Y=V_{X(\F_q)}(I(Y))$ by \cite[Lemma 2.8]{sahin18} which is nothing but $V_{X(\F_q)}(I_L)$ by above for a unique lattice $L$ with $(q-1)L_{\beta}\subseteq L\subseteq L_{\beta}$. 
\end{proof}

Therefore, in order the compute the order of a subgroup of the torus $T_X(\F_q)$ of $X$, it is enough to focus on lattices with $(q-1)L_{\beta}\subseteq  L\subseteq L_{\beta}$.

\begin{lm} \label{l:inclusionIffExp} Let $L\subseteq L_{\beta}$ and $\rank L=n$. Then, $(q-1)L_{\beta}\subseteq  L$ if and only if $\exp(L_{\bb}/L)=d_n$ divides $q-1$. Recall that the exponent $\exp(H)$ of a finite group $H$ is the least common multiple of the orders of all elements in $H$.
\end{lm}

\begin{proof} Recall that $L_{\beta}$ has a basis $\{\m_1,\dots,\m_n\}$ such that $\{d_1\m_1,\dots,d_n\m_n\}$ is a basis of $L$, where $d_i$'s are the invariant factors of $L_{\bb}/L$. By the structure theorem for finite abelian groups $L_{\bb}/L$ and $\Z_{d_1}\oplus \cdots \oplus \Z_{d_n}$ are isomorphic. Since $d_i$ divides $d_{i+1}$ for each $i\in [n-1]$, the exponent of $L_{\bb}/L$ is $d_n$. 

Therefore, if $(q-1)L_{\beta}\subseteq  L$, then $(q-1)\m_i\in L$. Since $\{d_1\m_1,\dots,d_n\m_n\}$ is a basis of $L$, it follows immediately that $d_i$ divides $q-1$, for every $i\in [n]$. Thus, $\exp(L_{\bb}/L)$ divides $q-1$. Conversely, if the exponent divides $q-1$, then each $d_i$ divides $q-1$, and hence $(q-1)\m_i\in L$ yielding to the inclusion $(q-1)L_{\beta}\subseteq  L$.
\end{proof}

 We are now ready to prove the main result of the section yielding to an efficient method to count the number of $\F_q$-rational points of the subgroup $V_X(I_L)$ of $T_X$ using a basis of $L$.

\begin{tm}\label{t:rationalpoints} If $(q-1)L_{\beta}\subseteq  L\subseteq L_{\beta}$, then the number of $\F_q$-rational points is given by $|V_{X(\F_q)}(I_L)|=|L_{\beta}/L|=d_1\cdots d_n$. 
\end{tm}
\begin{proof} By Lemma \ref{l:inclusionIffExp}, it follows that $d_i$ divides $q-1$, and hence $(d_i,q-1)=d_i$, for every $i\in [n]$. Therefore,  $|V_{X(\F_q)}(I_L)|=d_1\cdots d_n$, by Theorem \ref{t:Fq-rationalpoints}. 
\end{proof}

\begin{rema} \label{r:barker} Theorem \ref{t:rationalpoints} follows also from Lemma \ref{l:Subgroup=characters} together with a more general fact which is communicated in private to us by Laurence Barker: 
given a finite
group $A$ and a finite field $\F_q$, 
$\Hom(A, \F_q^*)$ is isomorphic to $A$ if and only if $\exp(A)$ divides $q-1$. The proof provided by Barker is as follows. 

The usual dual of $A$, sometimes called the
Pontryagin dual, is defined to be
$A^* = \Hom(A, \Q/\Z)$. The group $\Q/\Z$, is isomorphic to the torsion subgroup of
$\mathbb C^*$. So $A^*$
can be identified with the group of
irreducible complex characters of $A$.

The standard result that $A$ is
non-canonically isomorphic to $A^*$ is plain when $A$ is cyclic. It is easy to
see that $(A_1 \times A_2)^* \cong A_1^* \times A_2^*$.
So the general case follows by decomposing $A$
as a direct product of cyclic groups. 

Viewing finite cyclic groups as embedded in
each other wherever possible, then $\Q/\Z$ is
the union of all those groups, in other words,
the colimit of the diagram of embeddings.
In particular, $A^*$ can be identified with
$\Hom(A, C)$ for any cyclic group $C$ that
is sufficiently large in the sense where
$C' \leq C$ provided $|C'|$ divides
$|C|$. The exact criterion for $C$ to be
sufficiently large is that every homomorphism
$A \rightarrow \Q/\Z$ has image contained in
the subgroup isomorphic to $C$. That is
equivalent to saying that the exponent of
$A$ divides $|C|$. Taking $C=\F_q^*$ completes the proof.
\end{rema}

\begin{coro} \label{c:order=Det} If $(q-1)L_{\beta}\subseteq  L\subseteq L_{\beta}$, then $|V_{X(\F_q)}(I_L)|=|\det  \verb|ML||$, where $ \verb|ML|$ is the matrix with entries $b_{ij}\in \Z$ such that the $j$-th basis element of $L$ is written as $b_{1j}\uu_1+\cdots+b_{nj}\uu_n$ in terms of the basis $\{\uu_1,\dots,\uu_n\}$ of $L_{\beta}$.
\end{coro}
\begin{proof} By Theorem \ref{t:rationalpoints}, we have $|V_{X(\F_q)}(I_L)|=|L_{\beta}/L|$. Recall that the map $\phi$ of (\ref{eq:SES}) is an isomorphism between $\Z^n$ and $L_{\bb}$. If $\Lambda$ is the sublattice of $\Z^n$ which is isomorphic to $L$ under $\phi$, then the order $|L_{\beta}/L|=|\Z^n/\Lambda|$. 
Let $\verb|ML| :=\{\b_1,\dots,\b_n\}$ be a subset of $\Z^n$ constituting a $\Z$-basis for the lattice $\Lambda$. By abusing the notation, let us denote by  
$ \verb|ML|$ the matrix whose columns are the elements of the set $ \verb|ML|$.

It is known that the volume of the fundamental domain $\Pi( \verb|ML|)$ of $ \verb|ML| $ is $|\det  \verb|ML||$, where
$$\dis \Pi( \verb|ML|):=\Big\{\sum_{i=1}^{n} \lambda_i \b_i : 0 \leq \lambda_i < 1\Big\}.$$ 
 It is also known that $|\det  \verb|ML||$ is the number $|\Pi( \verb|ML|)\cap \Z^n|$ of lattice points inside $\Pi( \verb|ML|)$ as well as the index $[\Z^n:\Lambda]=|\Z^n/\Lambda|$.

Altogether, we have $|V_{X(\F_q)}(I_L)|=|L_{\beta}/L|=|\Z^n/\Lambda|=|\det  \verb|ML||$.
\end{proof}

\begin{coro}\label{c:c_i's} If $\{c_1\uu_1,\dots,c_n\uu_n\}$ constitute a basis for $L$, for positive integers $c_i$ dividing $q-1$, then $|V_{X(\F_q)}(I_L)|=c_1\cdots c_n$.
\end{coro}
\begin{proof} Since $\{\uu_1,\dots,\uu_n\}\subset \Z^r$ form a basis of $L_{\bb}$ and $\phi(\e_i)=\uu_i$, we observe that $\phi(c_i\e_i)=c_i\uu_i$, for all $i\in [n]$. So, $\verb|ML|=\{c_1\e_1,\dots,c_n\e_n\}$ is a basis for the lattice $\Lambda$ that appeared in the proof of Corollary \ref{c:order=Det} and $|\det(\verb|ML|)|=c_1\cdots c_n$.
Therefore, it follows that $\ell=|L_{\bb}/L|=c_1\cdots c_n$ and that $V_{X(\F_q)}(I_L)$ has $c_1\cdots c_n$ elements by Corollary \ref{c:order=Det}.
\end{proof}

 If $(q-1)L_{\beta}\subseteq  L\subseteq L_{\beta}$, then obviously $\ell=|L_{\bb}/L|$ is coprime to $p=\Char(\F_q)$. On the other hand, whenever $\ell=|L_{\bb}/L|$ is not divisible by a prime $p$, there are finite fields of characteristic $p$ verifying this condition as we demonstrate next.
 
\begin{pro} \label{p:existenceOfFq} Let $L\subseteq L_{\bb}$ and $\rank L=n$. If $\ell=|L_{\bb}/L|$ is not divisible by a prime $p$, then there is a finite field $\F_q$ of characteristic $p$ such that $\exp(L_{\bb}/L)$ divides $q-1$. Thus, we have $(q-1)L_{\bb}\subseteq L\subseteq L_{\bb}$ in this case.
\end{pro}
\begin{proof} Let $\F_q$ be the splitting field of $x^{\ell}-1$ over $\F_p$. As $p$ does not divide $\ell$, $x^{\ell}-1$ and its derivative $\ell x^{\ell-1}$ have no roots in common. So, roots of $x^{\ell}-1$ are all distinct. It is easy to see that they form a cyclic subgroup of $\F_q^*$ of order $\ell$. Thus, $\ell$ divides $q-1$ whence so does $\exp(L_{\bb}/L)$. 

The second part follows from Lemma \ref{l:inclusionIffExp}.
\end{proof}

We close the section with an interesting example illustrating the results.
\begin{ex}
\label{ex:15elements} Consider the Hirzebruch surface $X=\cl H_{2}$ whose fan in $\R^2$ have rays generated by 
$$\vv_1=(1,0), \vv_2=(0,1), \vv_3=(-1,2), \mbox{ and }\: \vv_4=(0,-1).$$ 

 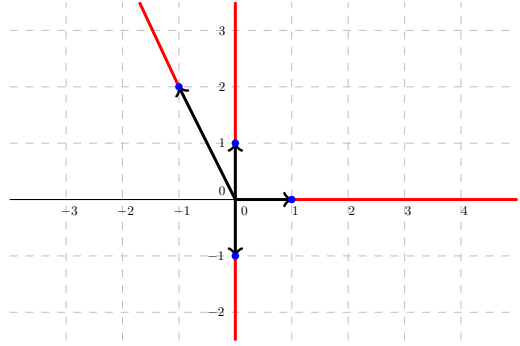
\begin{figure}[htb]
\centering
\begin{tikzpicture}[scale=0.75]
\draw (-4,0)  -- (5,0);
\draw (0,-2.5)  -- (0,3.5);
\draw[dashed,lightgray][line width=0mm] (-3,-2.5)  -- (-3,3.5);
\draw[dashed,lightgray][line width=0mm] (-2,-2.5)  -- (-2,3.5);
\draw[dashed,lightgray][line width=0mm] (-1,-2.5)  -- (-1,3.5);
\draw[dashed,lightgray][line width=0mm] (1,-2.5)  -- (1,3.5);
\draw[dashed,lightgray][line width=0mm] (2,-2.5)  -- (2,3.5);
\draw[dashed,lightgray][line width=0mm] (3,-2.5)  -- (3,3.5);
\draw[dashed,lightgray][line width=0mm] (4,-2.5)  -- (4,3.5);
\draw[dashed,lightgray][line width=0mm] (-4,1)  -- (5,1);
\draw[dashed,lightgray][line width=0mm] (-4,2)  -- (5,2);
\draw[dashed,lightgray][line width=0mm](-4,3)  -- (5,3);
\draw[dashed,lightgray][line width=0mm] (-4,-1)  -- (5,-1);
\draw[dashed,lightgray][line width=0mm] (-4,-2)  -- (5,-2);
\draw[thick,->][line width=0.4mm] (0,0) -- (-1,2);
\draw[thick,->][line width=0.4mm] (0,0) -- (0,1);
\draw[thick,->][line width=0.4mm] (0,0) -- (1,0);
\draw[thick,->][line width=0.4mm] (0,0) -- (0,-1);
\draw[red][line width=0.4mm] (-1,2)  -- (-1.7,3.5);
\draw (-1,2) node[circle,fill,blue,inner sep=1pt] {};
\draw[red][line width=0.4mm] (0,1)  -- (0,3.5);
\draw (0,1) node[circle,fill,blue,inner sep=1pt] {};
\draw[red][line width=0.4mm] (1,0)  -- (5,0);
\draw (1,0) node[circle,fill,blue,inner sep=1pt] {};
\draw[red][line width=0.4mm] (0,-1)  -- (0,-2.5);
\draw (0,-1) node[circle,fill,blue,inner sep=1pt] {};
\node[inner sep=0,anchor=west,{scale=0.5}] (note1) at (0.1,-0.2) {$0$};
\node[inner sep=0,anchor=west,{scale=0.5}] (note1) at (-0.3,0.15) {$0$};
\node[inner sep=0,anchor=west,{scale=0.5}] (note1) at (1,-0.2) {$1$};
\node[inner sep=0,anchor=west,{scale=0.5}] (note1) at (2,-0.2) {$2$};
\node[inner sep=0,anchor=west,{scale=0.5}] (note1) at (3,-0.2) {$3$};
\node[inner sep=0,anchor=west,{scale=0.5}] (note1) at (4,-0.2) {$4$};
\node[inner sep=0,anchor=west,{scale=0.5}] (note1) at (-1.1,-0.2) {$-1$};
\node[inner sep=0,anchor=west,{scale=0.5}] (note1) at (-2.1,-0.2) {$-2$};
\node[inner sep=0,anchor=west,{scale=0.5}] (note1) at (-3.1,-0.2) {$-3$};
\node[inner sep=0,anchor=west,{scale=0.5}] (note1) at (-0.3,1) {$1$};
\node[inner sep=0,anchor=west,{scale=0.5}] (note1) at (-0.3,2) {$2$};
\node[inner sep=0,anchor=west,{scale=0.5}] (note1) at (-0.3,3) {$3$};
\node[inner sep=0,anchor=west,{scale=0.5}] (note1) at (-0.5,-1) {$-1$};
\node[inner sep=0,anchor=west,{scale=0.5}] (note1) at (-0.5,-2) {$-2$};

\end{tikzpicture}
\caption{The fan for the Hirzebruch surface.}
\label{F:FanForH2}
%\caption{A sequence of bisections.}
%\label{fig:bisection}
\end{figure}

The exact sequence in (\ref{eq:SES}) becomes:
$$\dis \xymatrix{ \mathfrak{P}: 0  \ar[r] & \Z^2 \ar[r]^{\phi} & \Z^4 \ar[r]^{\beta}& \Z^2 \ar[r]& 0},$$  
where $$\phi=\begin{bmatrix}
1 & 0 & -1& ~~0 \\
0 & 1 & ~~~2& -1
\end{bmatrix}^T   \quad  \mbox{ and} \quad \beta=\begin{bmatrix}
1 & -2 & 1& 0 \\
0 & ~~~1 & 0&1  
\end{bmatrix}.$$ 
Thus, the class group is $\cl A=\Z^2$ grading the Cox ring  $S=\F_q[x_1,x_2,x_3,x_4]$ of $X$ via $$\deg_{\cl A}(x_1)=\deg_{\cl A}(x_3)=(1,0), \deg_{\cl A}(x_2)=(-2,1), \deg_{\cl A}(x_4)=(0,1).$$
Clearly, a $\Z$-basis of $L_{\bb}$ is given by $\uu_1=(1,0,-1,0)$ and $\uu_2=(0,1,2,-1)$. 

Consider the lattice $L$ spanned over $\Z$ by $3\uu_1$ and $5\uu_2$. Since $B=\{3\e_1,5\e_2\}$ is a basis for the lattice $\Lambda$ that appeared in the proof of Corollary \ref{c:order=Det} we have $\ell=|L_{\bb}/L|=\det B=15$ and so $15L_{\bb}\subset L \subset L_{\bb}$ by Lemma \ref{l:(ell*satL)SubsetL}. Hence, if $15$ divides $q-1$, which is the case for instance when $q=16$ or $q=31$, then the subgroup $V_{X(\F_q)}(I_L)$ has $15$ members of $T_X$ over the finite field $\F_q$.  
\end{ex}

\section{Finite Nullstellensatz for lattice ideals}

In this section, we prove a Nullstellensatz type result over $\F_q$ establishing a one-to-one correspondence between the subgroups $V_{X(\F_q)}(I_L)$ and lattice ideals $I_L$,  such that $(q-1)L_{\beta}\subseteq  L\subseteq L_{\beta}$.

\begin{lm}\label{l:Subgroups=overClosure} If $(q-1)L_{\beta}\subseteq  L\subseteq L_{\beta}$, then the subgroups $V_{X(\K)}(I_L)$ and $V_{X(\F_q)}(I_L)$ are the same, for any field $\K$ which is an extension of $\F_q$.
\end{lm}
\begin{proof} Since $\F_q^* \subset \K^*$, we have the inclusion $V_{X(\F_q)}(I_L)\subseteq V_{X(\K)}(I_L)$. By Theorem \ref{t:rationalpoints}, $|V_{X(\F_q)}(I_L)|=d_1\cdots d_n$, and so $\F_q^*$ has all the $d_i$-th roots of unity. Thus, $\K^*$ has all the $d_i$-th roots of unity and hence $|V_{X(\K)}(I_L)|=d_1\cdots d_n$ as well, by Theorem \ref{t:K-rationalpoints}. Having the same orders, $V_{X(\K)}(I_L)=V_{X(\F_q)}(I_L)$ for any field $\K$ which is an extension of $\F_q$. 
\end{proof}

It is now time to demonstrate that the necessary condition suggested by the Lemma \ref{l:latticeOFsubgroup} for a lattice to correspond to a vanishing ideal of a subgroup is indeed sufficient.
\begin{tm}\label{t:nullstellensatz} If $(q-1)L_{\beta}\subseteq  L\subseteq L_{\beta}$, then $I(V_{X(\K)}(I_L))=I_L$, for any field $\K$ which is an extension of $\F_q$.
\end{tm}
\begin{proof} By the same reason in the proof of \cite[Theorem 5.1(iii)]{sahin18}, we have the following minimal primary decomposition for the vanishing ideal of $V_{X(\K)}(I_L)$ in $S=\F_q[x_1,\dots,x_r]$:
$$\dis I(V_{X(\K)}(I_L))=\bigcap_{[P]\in V_{X(\K)}(I_L)} I([P]).$$

Suppose now that $\fp\subset S$ is a minimal prime of the radical ideal $I_L\subseteq I(V_{X(\K)}(I_L))$. Then, as $\fp \supseteq I_L$, we have $V_{X(\overline{\F}_q)}(\fp) \subseteq V_{X(\overline{\F}_q)}(I_L)$. As $V_{X(\overline{\F}_q)}(\fp)$ cannot be empty over the algebraically closed field $\overline{\F}_q$, it has a point $[P]\in V_{X(\overline{\F}_q)}(I_L)$. Since $V_{X(\overline{\F}_q)}(I_L)=V_{X(\K)}(I_L)=V_{X(\F_q)}(I_L)$, by Lemma \ref{l:Subgroups=overClosure}, $[P]\in V_{X(\K)}(I_L)$ as well. Therefore, $[P]\in V_{X(\K)}(\fp)$ so that $\fp \subseteq I(V_{X(\K)}(\fp)) \subseteq I([P]) \subset S$. So, $\h(\fp) \leq \h(I([P]))=n$. Since we also have $\fp \supseteq I_L$, it follows that $\h(\fp) \geq \h(I_L)=n$. Thus, $\h(\fp) = \h(I([P]))=n$ and these primes must coincide:  $\fp=I([P])$. 
Therefore, the two ideals $I_L \subseteq I(V_{X(\K)}(I_L))$ having the same minimal primary decomposition must coincide: $I_L= I(V_{X(\K)}(I_L))$.
\end{proof}

Recall that subgroups of $T_{X}(\F_q)$ are of the form $Y_Q=\{[{\bf t}^{\q_{1}}:\cdots:{\bf t}^{\q_{r}}]|{\bf t}\in (\F_q^{*})^s\}$ for a matrix $Q=[\q_1 \q_2\cdots \q_r]\in M_{s\times r}(\Z)$ by \cite[Theorem 3.2 and Corollary 3.7]{sahin18}. In \cite{EsmaTJM22}, a handier description of the lattice of the ideal $I(Y_Q)$ is given, in terms of $Q$ and $\beta$, under a condition on the lattice $\mathcal{L}=Q L_{\beta}=\{Q\m| \m\in L_{\beta}\}$. Before stating this result, let us remind that $\mathcal{L}:(q-1)=\{\m\in {\Z}^s|(q-1)\m \in \mathcal{L}\}$ and that the lattice $L_Q\subseteq \Z^r$ is the kernel of the multiplication map defined by $Q$, i.e. $L_Q=\{z\in \Z^r \: : \: Qz=0\}$. 

\begin{tm}[\cite{EsmaTJM22}] \label{t:Baranhold}
Let ${L}=(L_Q \cap {L_\beta})+(q-1)L_\beta$. Then 
$I_{L}\subseteq I(Y_Q)$. The equality holds if and only if $\mathcal{L}=\mathcal{L}:(q-1)$.
\end{tm}

If $L_Q \subseteq L_\beta$, then $I_{L_Q}$ is an $L_{\beta}$-homogeneous toric ideal whose zero locus inside the torus is denoted $V_Q:=V_{X(\F_q)}(I_{L_Q})\cap T_X$. The following consequence of Theorem \ref{t:Baranhold} ensures that $V_Q$ coincides with the subgroup $Y_Q$ if the condition $\mathcal{L}=\mathcal{L}:(q-1)$ holds, following \cite[Proposition 4.3]{AlgebraicMethods2011} and \cite[Corollary 4.4]{AlgebraicMethods2011}. 

\begin{tm}[\cite{EsmaTJM22}] \label{t:BSnullstellensatz} If $L_Q \subseteq L_\beta$ and ${L}=L_Q +(q-1)L_\beta$, then we have the following
\begin{enumerate}
\item $V_Q=V_{X(\F_q)}(I_{L})$,
\item $V_Q=Y_Q$, if the condition $\mathcal{L}=\mathcal{L}:(q-1)$ holds,
\item $I(V_{X(\F_q)}(I_{L}))=I_{L}$, if the condition $\mathcal{L}=\mathcal{L}:(q-1)$ holds.
\end{enumerate} 
\end{tm}

Using Theorem \ref{t:nullstellensatz}, we are now able to prove  that $V_Q$ coincides with the subgroup $Y_Q$ if and only if the condition $\mathcal{L}=\mathcal{L}:(q-1)$ holds. 

\begin{coro} If $L=(L_Q\cap L_{\bb})+(q-1)L_{\bb}$, then  $Y_Q=V_{X(\F_q)}(I_L)$ if and only if $\mathcal{L}:(q-1)=\cl L$.
\end{coro}
\begin{proof} If $\mathcal{L}:(q-1)=\cl L$ then by Theorem \ref{t:Baranhold} we have $I(Y_Q)=I_L$ implying that $Y_Q=V_{X(\F_q)}(I(Y_Q))=V_{X(\F_q)}(I_L)$. For the converse it suffices to note that the lattice $L=(L_Q\cap L_{\bb})+(q-1)L_{\bb}$ satisfies the condition $(q-1)L_{\beta}\subseteq  L\subseteq L_{\beta}$ of Theorem \ref{t:nullstellensatz}, and hence $I(Y_Q)=I(V_{X(\F_q)}(I_L))=I_L$. Applying Theorem \ref{t:Baranhold} again, we get $\mathcal{L}:(q-1)=\cl L$.
\end{proof}

We close this section by sharing a procedure for computing the vanishing ideal of $V_{X(\F_q)}(I_L)\cap T_X$ for lattices $L\subseteq L_{\bb}$ not satisfying the condition $(q-1)L_{\bb}\subset L$.

\begin{coro}\label{c:I_L'} If $L\subseteq L_{\bb}$, then the vanishing ideal of $V_{X(\F_q)}(I_L)\cap T_X$ is $I_{L'}$, for the lattice $L'=L+(q-1)L_{\bb}$. Moreover, the order of $V_{X(\F_q)}(I_L)\cap T_X$ is $|L_{\bb}/L'|$.
\end{coro}
\begin{proof} If $L'=L+(q-1)L_{\bb}$, then $I_{L'}=I_L+I_{(q-1)L_\bb}$. Since $I(T_X(\F_q))=I_{(q-1)L_\bb}$ it follows that we have
$$V_{X(\F_q)}(I_{L'})=V_{X(\F_q)}(I_L)\cap V_{X(\F_q)}(I_{(q-1)L_\bb})=V_{X(\F_q)}(I_L)\cap T_X.$$ 
As $(q-1)L_{\bb}\subset L' \subset L_{\bb}$, Theorem \ref{t:nullstellensatz} implies that $I_{L'}=I(V_{X(\F_q)}(I_{L'}))$. Hence, the $B$-saturated ideal corresponding to $Y=V_{X(\F_q)}(I_L)\cap T_X$ is $I(Y)=I_{L'}$.
The second claim follows directly from Theorem \ref{t:rationalpoints}.
\end{proof}

\section{Degree of a lattice ideal}
In this section, we assume that the $n$-dimensional toric variety $X$ is smooth and projective. Let $S$ be the Cox ring of $X$ and $M$ an $S$-graded module. Denote by $H_M(\aa)$ the multigraded Hilbert function of $M$, that is, $H_M(\aa)=\dim_{\K}(M_{\aa})$. Maclagan and Smith proved in \cite[Proposition 2.10]{MS05-MultHPolynomial} that there is a polynomial $P_M(t_1,\dots,t_d)\in \Q[t_1,\dots,t_d]$ called the \textit{multigraded Hilbert polynomial} of $M$ such that $P_M(\aa)=H_M(\aa)$ for a non-empty subset of $\N\bb$. They also showed that for a multigraded ideal $I$, the Hilbert polynomials of $S/I$ and that of $S/(I:B^{\infty})$ are the same, where $(I:B^{\infty})$ is the $B$-saturation of $I$, see \cite[Lemma 2.13]{MS05-MultHPolynomial}. Therefore, if $I$ is the $B$-saturated ideal corresponding to a finite set of points then its Hilbert polynomial is constant and equal to the cardinality of $V_X(I)$, see \cite[Example 4.12]{MS05-MultHPolynomial}. Inspired from these, we make the following
\begin{defi}
Let $I$ be the $B$-saturated homogeneous ideal corresponding to a finite set of points. The degree of $I$ is the constant Hilbert polynomial $P_{S/I}$.
\end{defi}

\begin{lm} \label{l:HPofI_LoverK} If $I_L$ is an $L_{\bb}-$homogeneous ideal, then its (multigraded) Hilbert function defined by $H_{I_L}(\aa):=\dim_{\K}(S_{\aa}/I_L)$ is independent of the field $\K$. 
\end{lm}
\begin{proof} Recall that the vector space $S_{\aa}$ is spanned by the monomials $$\x^{\a}=x_1^{a_1}\cdots x_r^{a_r} \mbox{ with } \deg(\x^{\a})=\sum_{j=1}^{r}a_j\bb_j=\aa.$$ By \cite[Theorem 7.3]{CombComAlgBook}, the quotient ring $S/I_L$ is isomorphic to the semigroup ring of the semigroup $\N^r/\sim_L$, where $\a_1 \sim_L \a_2 \iff \a_1-\a_2 \in L$. Since $I_L$ is $L_{\bb}-$homogeneous, we have $L\subseteq L_{\bb}$. Therefore, $\a_1-\a_2 \in L$ implies $\a_1-\a_2 \in L_{\bb}$, i.e. if $\x^{\a_1}+I_L=\x^{\a_2}+I_L$ in $S/I_L$, then $\deg(\x^{\a_1})=\deg(\x^{\a_2})$ meaning that $\x^{\a_1}+I_L=\x^{\a_2}+I_L$ in $S_{\aa}/I_L$. Hence, $H_{I_L}(\aa)=\dim_{\K}(S_{\aa}/I_L)$ is the number of equivalence classes $\x^{\a}+I_L$ of monomials $\x^{\a}$ of degree $\aa\in \N_{\bb}$, which is also the number of equivalence classes of vectors $\a \in \N^r$ with respect to the equivalence relation $\sim_L$ such that $\sum_{j=1}^{r}a_j\bb_j=\aa$. This number is independent of the field, completing the proof. 
\end{proof}

\begin{tm} \label{t:DegreeOfI_L}
 If $I_L$ is an $L_{\bb}-$homogeneous ideal of dimension $d=r-n$, then its degree is $\deg(I_L)=|L_{\bb}/L|=d_1\cdots d_n$, where $d_i$'s are the invariant factors of $L_{\bb}/L$.
\end{tm}
\begin{proof} By Lemma \ref{l:HPofI_LoverK}, the multigraded Hilbert polynomial of $I_L$ is independent of the field. Under the hypothesis $\rank L=n$ and $L\subseteq L_{\bb}$, so by Proposition \ref{p:existenceOfFq} there is a finite field $\F_q$ such that $(q-1)L_{\beta}\subseteq  L\subseteq L_{\beta}$. Thus, the number of $\F_q$-rational points is $|V_X(I_L)(\F_q)|=d_1\cdots d_n$ by Theorem \ref{t:rationalpoints}. On the other hand, the $B$-saturated ideal corresponding to $V_X(I_L)$ is $I(V_X(I_L))=I_L$ by Theorem \ref{t:nullstellensatz}. Thus, the Hilbert polynomial of $I_L$ will be $d_1\cdots d_n$, which is the degree of $I_L$.
\end{proof}

The following reveals that Theorem \ref{t:DegreeOfI_L} generalizes the main result of the nice paper \cite{HLRV2014} by Lopez and Villarreal. $X$ is the projective space $\mathbb{P}^n$ in this case. The matrix $\bb=[1 \cdots 1]$, $L_{\bb}=\Z\{e_1-e_r,\dots,e_{r-1}-e_r\}$ and hence an ideal is $L_{\bb}-$ homogeneous with respect to the standard grading: $\deg (x_i)=1$, for each $i\in [r]$, where $r=n+1$.

\begin{coro} \cite[Theorem 3.13]{HLRV2014}  If $I_L$ is an $L_{\bb}-$ homogeneous ideal of dimension $1$, then its degree is $\deg(I_L)=|T(\Z^r /L)|$, where $T(\Z^r /L)$ is the torsion part of the group $\Z^r /L$.
\end{coro}
\begin{proof} It is well known from the fundamental structure theorem for finitely generated abelian groups that the order of the torsion part $T(\Z^r /L)$ of the group $\Z^r /L$ is $d_1\cdots d_n$, where $d_i$'s are the invariant factors of the matrix whose columns span the lattice $L$, see \cite[pp. 187–-188]{Jacobson}. So, Theorem \ref{t:DegreeOfI_L} completes the proof. 
\end{proof}

If an ideal $I=\la F_1,\dots,F_n \ra$ is a complete intersection generated in semi-ample degrees $\aa_1,\dots,\aa_n$, then its Hilbert Polynomial is the Bernstein–-Kushnirenko bound, which is the mixed volume $n!V(P_{\aa_1},\dots,P_{\aa_n})$ of the Newton polytopes $P_{\aa_1},\dots,P_{\aa_n}$ by the proof of \cite[Theorem 3.16]{MultHFuncToricCodes16}. The following reveals the same is true in the special case that $I=I_L$ is a lattice ideal, without the assumption that the degrees $\aa_1,\dots,\aa_n$ are semi-ample.

\begin{coro} \label{c:mixedvolume}Assume that $\{d_1\m_1,\dots,d_n\m_n\}$ is a basis of the lattice $L$ such that $\{\m_1,\dots,\m_n\}$ is a basis of $L_{\bb}$. If $P_i$ is the Newton polytope of the binomial $F_i=\x^{d_i\m_i^{+}}-\x^{d_i\m_i^{-}}$, for $i\in [n]$, then the degree of $I_L$ is the mixed volume $n!V(P_{1},\dots,P_{n})$.
\end{coro}
\begin{proof} By Proposition \ref{p:existenceOfFq}, there is a finite field $\K=\F_q$ containing all the $d_i$-th roots of unity, for each $i \in [n]$. We know that $V_X(I_L)(\K)$ lies inside the torus $T_X(\K)$ by Proposition \ref{p:VX(IL)lyingInTX}. So, localizing to the torus does not change the order of $V_X(I_L)(\K)$. This amounts to passing to the Laurent polynomial ring $$\K[\x^{\mp}]=\K[x_1,\dots,x_r]_{x_1\cdots x_r}=\K[x_1,\dots,x_r,x_1^{-1},\dots,x_r^{-1}].$$ By \cite[Theorem 2.1]{BinomialIdeals}, the ideal $\K[\x^{\mp}]\cdot I_L$ is a local complete intersection generated by the Laurent polynomials $\x^{d_1\m_1}-1, \dots,\x^{d_n\m_n}-1$ forming a regular sequence, even if $I_L$ is not the (global) complete intersection ideal $$\la \x^{d_1\m_1^{+}}-\x^{d_1\m_1^{-}}, \dots, \x^{d_n\m_n^{+}}-\x^{d_n\m_n^{-}}\ra.$$ Letting $t_i:=\x^{\m_i}$ be the local variables for the torus $T_X(\K)\cong (\K^*)^n$, we get the system $t_1^{d_1}=\cdots=t_n^{d_n}=1$ having $d_1\cdots d_n$ solutions as $\K$ has all the $d_i$-th roots of unity, for each $i \in [n]$. Thus, $|V_X(I_L)(\K)|$ meets the Bernstein-–Kushnirenko bound $n!V(P_{\aa_1},\dots,P_{\aa_n})=d_1\cdots d_n$  by \cite{HuStu95} as $P_i$ are just line segments, see also \cite{Bernstein,Kushnirenko}. Hence,  Theorem \ref{t:DegreeOfI_L} completes the proof.
\end{proof}

\begin{ex} \label{ex:15ElementsDegreeIs5} We revisit Example \ref{ex:15elements}. Consider the fan in $\R^2$ with rays generated by $\vv_1=(1,0)$, $\vv_2=(0,1)$, $\vv_3=(-1,2)$,
and $\vv_4=(0,-1)$. Then the corresponding toric variety is the Hirzebruch surface $X=\cl H_{2}$. Recall that a $\Z$-basis of $L_{\bb}$ is given by $\uu_1=(1,0,-1,0)$ and $\uu_2=(0,1,2,-1)$. Consider the lattice $L$ spanned over $\Z$ by $3\uu_1$ and $5\uu_2$. Then, clearly $15L_{\bb}\subset L \subset L_{\bb}$. So, the degree of $I_L$ is just the order $|L_{\bb}/L|=15$. 

If $p$ is a prime number other than $3$ or $5$, then by Proposition \ref{p:existenceOfFq} there is a finite field $\F_q$ whose characteristic is $p$ for which $(q-1)L_{\bb}\subset L \subset L_{\bb}$ is satisfied and thus $I_{L}=I(V_{X(\F_q)}(I_L))$ is the $B$-saturated ideal of $V_{X(\F_q)}(I_L)$, by Theorem \ref{t:nullstellensatz} over the field $\F_{q}$, where $B=\langle x_1x_2,x_2x_3,x_3x_4,x_4x_1 \rangle$. This is the case for instance with $p=2$, and thus, $V_{X(\F_q)}(I_{L})$  has $15=\deg(I_{L})$ members of $T_X$ over $\F_{16}$. Similarly, $I_L$ is the $B$-saturated ideal of $V_{X(\F_q)}(I_L)$ over the finite field $\F_q$ of characteristic $p=11$, with $q=11^2=121$ elements, as $15$ divides $120$. $V_{X(\F_q)}(I_{L})$  has $15=\deg(I_{L})$ members of $T_X$ over $\F_{121}$.

However, if $q=p=11$ then $I_{L}$ is not the $B$-saturated ideal of $V_{X(\F_q)}(I_L)$ as $15$ does not divide $10$. Over the field $\F_{11}$, we have to replace $L$ by $L'=L+(q-1)L_{\bb}$ so as to obtain the correct lattice whose ideal $I_{L'}$ is the $B$-saturated vanishing ideal of $V_{X(\F_q)}(I_L)$, by Corollary \ref{c:I_L'}. A basis for the lattice $L'$ is given by $\uu_1$ and $5\uu_2$, since $\uu_1=-3(3\uu_1)+10\uu_1 \in L'$. The Smith-normal form of the matrix with columns $\uu_1$ and $5\uu_2$ is $[\e_1 | 5 \e_2]$, where $\e_i$ is a standard basis vector of $\Z^4$. Thus, $V_{X(\F_q)}(I_{L})$  has $5=\deg(I_{L'})$ members of $T_X$ over the finite field $\F_{11}$.
\end{ex}

\section{Parameterization of the subgroup cut out by a lattice ideal}
In this section, we describe the points of the subgroup $V_{X(\K)}(I_L)$ of the torus $T_X$ defined by the lattice ideal $I_L$.

Let $\cl B$ be an $r\times n$ matrix whose columns constitute a basis for $L\subseteq L_{\bb}$; let $\cl A$ and $\cl C$ be unimodular matrices of sizes $r$ and $n$, respectively, so that $$\cl D=\cl A \cl B \cl C=[d_1\e_1 |\cdots |d_n\e_n]$$ is the Smith-Normal form of $\cl B$, where $\e_1,\dots,\e_n \in \Z^r$. Recall that $d_i$ divides $d_{i+1}$, for all $i\in [n-1]$. If $\cl A^{-1}=[\m_1 |\cdots |\m_r]$ is the inverse of $\cl A$, then its columns span $\Z^r$. Thus, the columns $\{d_1\m_1,\dots,d_n\m_n\}$ of ${\cl A}^{-1}\cl D=\cl B \cl C$ provide us with another basis of the lattice $L$ such that $\{\m_1,\dots,\m_n\}$ is a basis of $L_{\bb}$. Since $\rank L=n$, we have $V_{X(\K)}(I_L)\subseteq T_X$ by Proposition \ref{p:VX(IL)lyingInTX}.

\begin{tm} \label{t:parameterization} Let $L\subseteq L_{\bb}$ be a lattice of rank $n$. Then, the following hold:
\begin{itemize}
\item[(i)] If $[P_i]=[\eta_i^{a_{i1}}:\cdots:  \eta_i^{a_{ir}}]$, where $\eta_i$ is a generator for the cyclic subgroup of $d_i$-th roots of unity in $\K$, for all $i\in [n]$, and $a_{ij}$'s are the entries of $\cl A$, then 
\[\dis V_X(I_L)(\K)=\la [P_1] \ra \times \dots \times \la [P_n] \ra.
\]
\item[(ii)] If $d_i$ divides $q-1$, for all $i\in [n]$, and  $Q$ is the matrix whose $i$-th row is $[a_{i1}(q-1)/d_i | \dots | a_{ir}(q-1)/d_i]$ which is $(q-1)/d_i$ times the $i$-th row of the matrix $\cl A$, then $V_X(I_L)(\F_q)=Y_Q$.
\end{itemize}
\end{tm}
\begin{proof}
The algebraic isomorphism induced by the change of basis map $\Z^r \rightarrow \Z^r$  defined via multiplication by the matrix $\cl A=[\a_1 | \cdots | \a_r]$ is as follows:
\begin{equation} \label{eq:Phi_A}
\Phi_{\cl A} : \T^r \rightarrow \T^r, \quad \Phi_{\cl A}(y_1,\dots,y_r)=(\y^{\a_1},\dots,\y^{\a_r}).
\end{equation} 
with the following inverse isomorphism:
\begin{equation} \label{eq:invPhi_A}
\Phi^{-1}_{\cl A} : \T^r \rightarrow \T^r, \quad \Phi^{-1}_{\cl A}(x_1,\dots,x_r)=(\x^{\m_1},\dots,\x^{\m_r}).
\end{equation} 
Recall that $G$ is the subgroup $V(I_{L_{\bb}})\cap (\overline{\K}^*)^r$ and thus $(x_1,\dots,x_r)\in G$ if and only if $\x^{\m_1}=\cdots=\x^{\m_n}=1$.
Therefore, if $(y_1,\dots,y_r)=\Phi^{-1}_{\cl A}(x_1,\dots,x_r)$, then it follows that $(x_1,\dots,x_r)\in G$ if and only if $y_1=\cdots=y_n=1$.
Thus, the group $G':=\Phi^{-1}_{\cl A}(G)$ is described by
\[G'= \{(1,\dots,1,y_{n+1},\dots,y_r) : y_i \in \T^r\}.\]

On the other hand, since a basis for $L$ is given by the set $\{d_1\m_1,\dots,d_n\m_n\}$, we similarly have the following:
\[(x_1,\dots,x_r)\in V(I_{L})\cap \T^r \iff \x^{d_1\m_1}=\cdots=\x^{d_n\m_n}=1.
\]
Finally, as the elements $[x_1:\cdots:x_r]\in T_X(\K)$ are orbits $G\cdot (x_1,\dots,x_r)$, it follows that
\[[x_1:\cdots:x_r]\in T_X(\K) \iff [y_1:\cdots:y_n:1:\cdots:1]\in \T^r /G'.
\]
Thus, $\Phi^{-1}_{\cl A}$ induces an isomorphism between $T_X(\K)$ and the torus $\T^r /G'$ identifying the subgroup $V_X(I_{L})(\K)$ of $T_X(\K)$ with the following subgroup:
\[\cl G_L:=\{ [y_1:\cdots:y_n:1:\cdots:1]\in \T^r /G': y_1^{d_1}=\cdots=y_n^{d_n}=1\}.
\]

If $C_i=\{y_i \in \K^* : y_i^{d_i}=1\}$, for all $i\in [n]$, then, $\cl G_L$ is isomorphic to $C_1\times \cdots \times C_n$, via
\[[y_1:\cdots:y_n:1:\cdots:1] \rightarrow (y_1,\dots,y_n).
\]

If $C_i$ is the cyclic subgroup of $\K^*$ generated by $\eta_i$, then the elements of $\cl G_L$ are of the form
\[[y_1:\cdots:y_n:1:\cdots:1]=[\eta_1^{k_1}:\cdots:\eta_n^{k_n}:1:\cdots:1],
\]
which bijectively correspond to the $n$-tuples $(k_1,\dots,k_n)$ with $0\leq k_i \leq |C_i|-1$. Let $[P]=[x_1:\cdots:x_r]$ be a point with $x_i=\y^{\a_i}=y_1^{a_{1i}}\cdots y_r^{a_{ri}}$. Then, we have
\[
\dis [P] \in V_X(I_L)(\K) \iff  \Phi^{-1}_{\cl A}(P)\in \cl G_L.
\]
Therefore, $[P] \in V_X(I_L)(\K) \iff $ there is an $n$-tuple $(k_1,\dots,k_n)$ with $0\leq k_i \leq |C_i|-1$ such that
\begin{eqnarray}
[P]&=&[\eta_1^{a_{11}k_1} \cdots \eta_n^{a_{n1}k_n} :\cdots:  \eta_1^{a_{1r}k_1} \cdots \eta_n^{a_{nr}k_n}] \label{e:point}\\
&=&[\eta_1^{a_{11}}:\cdots:  \eta_1^{a_{1r}}]^{k_1}\cdots [\eta_n^{a_{n1}}:\cdots:  \eta_n^{a_{nr}}]^{k_n}
= [P_1]^{k_1}\cdots [P_n]^{k_n}. \nonumber
\end{eqnarray}

This proves the first part (i):
\[\dis V_X(I_L)(\K)=\la [P_1] \ra \times \dots \times \la [P_n] \ra.
\]

When $\K=\F_q$ and $d_i$ divides $q-1$, $C_i$ becomes the cyclic subgroup of $\F_q^*=\la \eta  \ra$ of order $d_i$, generated by $\eta_i=\eta^{q-1/d_i}$, for all $i\in [n]$. Letting $t_i=\eta^{k_i}$ for all $i\in [n]$, we observe that the point in (\ref{e:point}) becomes
\begin{eqnarray*}
[P]&=&[\eta^{a_{11}k_1(q-1)/d_1} \cdots \eta^{a_{n1}k_n(q-1)/d_n} :\cdots:  \eta^{a_{1r}k_1(q-1)/d_1} \cdots \eta^{a_{nr}k_n(q-1)/d_n}]\\
&=&[t_1^{a_{11}(q-1)/d_1} \cdots  t_n^{a_{n1}(q-1)/d_n}:\cdots:  t_1^{a_{1r}(q-1)/d_1} \cdots t_n^{a_{nr}(q-1)/d_n}].
\end{eqnarray*}
If $Q$ is the matrix whose $i$-th row is $[a_{i1}(q-1)/d_i | \dots | a_{ir}(q-1)/d_i]$, then it follows that $[P] \in V_X(I_L)(\F_q)$ implies $[P] \in Y_Q$. 

In order to prove the converse inclusion, take $[P] \in Y_Q$. Then, there are $t_i \in \F_q^*$ such that
 \begin{eqnarray*}
[P]&=&[t_1^{a_{11}(q-1)/d_1} \cdots  t_n^{a_{n1}(q-1)/d_n}:\cdots:  t_1^{a_{1r}(q-1)/d_1} \cdots t_n^{a_{nr}(q-1)/d_n}].
\end{eqnarray*}
Since $\F_q^*=\la \eta \ra$, we have $\ell_i$ between $0$ and $q-2$ for which $t_i=\eta^{\ell_i}$, for all $i\in [n]$. Then, there is a unique $k_i$ between $0$ and $d_i-1$ such that $\ell_i \equiv k_i$ modulo $d_i$. Thus, $t_i^{(q-1)/d_i}=\eta^{\ell_i (q-1)/d_i}=\eta^{k_i (q-1)/d_i}$, yielding $[P] \in V_X(I_L)(\F_q)$.
\end{proof}

\begin{coro} \label{c:cyclicParameterization}If $V_X(I_L)(\F_q)$ is a cyclic subgroup, then $V_X(I_L)(\F_q)=Y_Q$ for a row matrix $Q=[a_{n1}(q-1)/d_n | \dots | a_{nr}(q-1)/d_n]$ which is $(q-1)/d_n$ times the $n$-th row of the matrix $\cl A$.
\end{coro}
\begin{proof} By Theorem \ref{t:rationalpoints}, $V_X(I_L)(\F_q)$ has $d_1\cdots d_n$ elements. Since it is cyclic, $d_1=\cdots=d_{n-1}=1$ and $|V_X(I_L)(\F_q)|=d_n$. By Theorem \ref{t:parameterization}, it follows that $V_X(I_L)(\F_q)=Y_Q'$, where $Q'$ is the matrix whose $i$-th row is $(q-1)/d_i$ times the $i$-th row of the matrix $\cl A$. When $d_i=1$, the following point 
\begin{eqnarray*}
\dis [P]=[t_1^{a_{11}(q-1)/d_1} \cdots  t_n^{a_{n1}(q-1)/d_n}:\cdots:  t_1^{a_{1r}(q-1)/d_1} \cdots t_n^{a_{nr}(q-1)/d_n}]
\end{eqnarray*}
of $Y_Q'$ becomes 
$$\dis [P]=[ t_n^{a_{n1}(q-1)/d_n}:\cdots:  t_n^{a_{nr}(q-1)/d_n}]$$ 
due to the fact that $t_i^{a_{ij}(q-1)/d_i}=1$, for all $i\in [n-1]$ and $j\in [r]$. Thus, $[P]$ is a point of $Y_Q$. As the converse is also true, the claim follows.
\end{proof}

\begin{rema} Although a parameterization $Y_Q$ of the subgroup $V_{X(\F_q)}(I_L)\cap T_X$ is given in \cite[Proposition 3.4]{sahin18}, the size of the matrix $Q$ for which $V_{X(\F_q)}(I_L)\cap T_X=Y_Q$ was always $r\times r$, even if the subgroup is cyclic. This is in contrast with the matrices provided in Theorem \ref{t:parameterization} and Corollary \ref{c:cyclicParameterization} that are more natural and succinct.
\end{rema}
We illustrate the main result of this section working with $\K=\R$ and $\K=\mathbb{C}$.
\begin{ex}
\label{ex:2real6complexSolutions} We revisit the example \ref{ex:15elements} by considering the Hirzebruch surface $X=\cl H_{2}$ over $\K=\R$ or $\K=\C$ this time. Consider the lattice $L$ spanned by the columns of the matrix $\cl B=[2\uu_1 \quad 3\uu_2]$. Then the Smith-Normal form $\cl D$ of $\cl B$  is given by 
\[
\cl D=\cl A \cl B \cl C=
\begin{bmatrix}
 0 & 0 & 1 & 1 \\  
 0 & 0 & 3 & 4 \\
 0 & 1 & 0 & 1\\
 1 & 0 & 1 & 2   
\end{bmatrix} 
\begin{bmatrix}
\phantom{-}2 & \phantom{-}0  \\  
 \phantom{-}0 & \phantom{-}3  \\ 
 -2 & \phantom{-}6  \\ 
 \phantom{-}0 & -3  
\end{bmatrix}
\begin{bmatrix}
 1 & -3  \\  
 1 & -2  \\ 
\end{bmatrix}
=\begin{bmatrix}
 1 & 0  \\  
 0 & 6  \\ 
 0 & 0  \\ 
 0 & 0  
\end{bmatrix}.
\]
The cyclic subgroup $C_1$ of the $1$-st roots of unity is trivial over $\R$ or $\C$. But, the cyclic subgroup $C_2$ of the $6$-th roots of unity are as follows respectively over $\R$ or $\C$.
\begin{eqnarray*}
C_2&=&\{c\in \R^* \: : \: c^6=1\}=\{-1,1 \},\\
C_2&=&\{c\in \C^* \: : \: c^6=1\}=\la \eta_2\ra = \la (1 + i\sqrt{3})/2 \ra \\
&=&\{1, (1 + i\sqrt{3})/2,(-1 + i\sqrt{3})/2,-1,(-1 - i\sqrt{3})/2,(1 - i\sqrt{3})/2 \}.   
\end{eqnarray*}
It follows from Theorem \ref{t:K-rationalpoints} that $|V_{X(\K)}(I_L)|=|C_1|\cdot|C_2|$ which is $2$ and $6$ over $\R$ and $\C$ respectively.
Indeed,
\[
V_{X(\K)}(I_L)=\{ [x_1:x_2:x_3:x_4] \in X(\K) \: : \: x_1^2=x_3^2 \text{ and } x_4^3=(x_2x_3^2)^3\}.
\]
Since $G=\{[\lambda_1:\lambda_2:\lambda_1:\lambda_1^{2}\lambda_2] \: : \: \lambda_1,\lambda_2\in \C^*\}$, we have
\begin{eqnarray*}
V_{X(\R)}(I_L)&=&\{[x_1:x_2:x_1:x_1^{2}x_2],[x_1:x_2:-x_1:x_1^{2}x_2] \: : \: x_3,x_4 \in \R^*\},\\
&=& \{[1:1:1:1],[1:1:-1:1] \}  \text{ and }  \\
V_{X(\C)}(I_L)&=&\{[1:1:\pm 1: 1],[1:1:\pm 1:\eta_2^2],[1:1:\pm 1:\eta_2^4]\}.
\end{eqnarray*} 
Notice that $V_{X(\C)}(I_L)=\la [P_1] \ra \times  \la [P_2] \ra$, where $[P_1]=[\eta_1^0:\eta_1^0:\eta_1^1:\eta_1^1]=[1:1:1:1]$ and $[P_2]=[\eta_2^0:\eta_2^0:\eta_2^3:\eta_2^4]=[1:1:-1:\eta_2^4]$, as stated in Theorem \ref{t:parameterization}.
The isomorphism between the groups $V_{X(\K)}(I_L)$ and $C_1 \times C_2$ is not clear at this stage. By following the steps in the proof of Theorem \ref{t:parameterization}, we shall now make this more precise. It is easy to see that
\[
\cl A^{-1}=[ \m_1 \quad \m_2 \quad  \m_3 \quad \m_4]=
\begin{bmatrix}
 \phantom{-}2 & -1 & 0 & 1 \\  
 \phantom{-}3 & -1 & 1 & 0 \\
 \phantom{-}4 & -1 & 0 & 0\\
 -3 & \phantom{-}1 & 0 & 0   
\end{bmatrix} 
.
\]
Therefore, the map $\Phi^{-1}_{\cl A} : \T^4 \rightarrow \T^4$ is given by
%\begin{equation} \label{eq:Phi_ex}
 %\Phi_{\cl A}(y_1,y_2,y_3,y_4)=(y_3,y_3^{-2}y_4,y_1y_2^3y_3,y_1y_2^4y_4)=(x_1,x_2,x_3,x_4).
%\end{equation} 
\begin{equation} \label{eq:Phi_ex}
 \Phi^{-1}_{\cl A}(x_1,x_2,x_3,x_4)=(x_1^{2}x_2^{3}x_3^{4}x_4^{-3},x_1^{-1}x_2^{-1}x_3^{-1}x_4^{},x_2,x_1)=(y_1,y_2,y_3,y_4).
\end{equation} 
 Hence, we have
\[
\Phi^{-1}_{\cl A}(V_{X(\K)}(I_L))={\cl G}_L(\K)=\{ [y_1:y_2:1:1] \in X(\K) \: : \: y_1^1=1 \text{ and } y_2^6=1\},
\]
which is canonically isomorphic to $C_1\times C_2$. Clearly, $$\Phi^{-1}_{\cl A}(1,1,\pm 1,\eta_2^j)=(1,\pm \eta_2^j,1,1), \text{ for all } j=0,2,4.$$ Thus, we have
\begin{eqnarray*}
\Phi^{-1}_{\cl A}(V_{X(\C)}(I_L))={\cl G}_L(\C)&=&\{[1:\pm 1:1:1],[1:\pm \eta_2^2:1:1],,[1:\pm \eta_2^4:1:1]\}.
\end{eqnarray*}
\begin{eqnarray*}
\text{ Similarly, we have } \Phi^{-1}_{\cl A}(V_{X(\R)}(I_L))={\cl G}_L(\R)=\{[1:1:1:1],[1:-1:1:1] \} .   \\
\end{eqnarray*}
\end{ex}

\section{Evaluation codes on subgroups of the torus $T_X(\F_q)$}

In this section, we apply the main results of the paper to compute basic parameters of toric codes on subgroups $Y=V_{X(\F_q)}(I_L)$ of the torus $T_X(\F_q)$, where a basis for $L$ is given by $\{c_1\uu_1,\dots,c_n\uu_n\}$ for positive integers $c_i$ dividing $q-1$.  

We first recall the evaluation code defined on a subset $Y=\{[P_1],\dots,[P_N]\}$ of $T_X(\F_q)$. Let $\N\beta$ be the subsemigroup of $\N^n$ generated (not necessarily minimally) by $\bb_1,\dots,\bb_r$. Recall that the Cox ring $\displaystyle S=\bigoplus_{\aa \in \N\bb}S_{\aa}$ of $X$ is multigraded by $\N\bb$ via $\deg (x_i)=\bb_i$ for $i=1,\dots,r$. For any multidegree $\aa\in\N\beta$, we have the following {\it evaluation map}
\begin{equation}\label{e:evalmap}
\dis \ev_{Y}:S_\aa\to \F_q^N,\quad F\mapsto \left(F(P_1),\dots,F(P_N)\right).
\end{equation}
The image $\cl{C}_{\aa,Y}=\text{ev}_{Y}(S_\aa)$ is a linear code, called the {\it generalized toric code}. The three basic parameters of $\cl C_{\aa,Y}$ are block-length which is $N$, the dimension which is $K=\dim_{\F_q}(\cl C_{\aa,Y})$, and the minimum distance $\delta=\delta(\cl{C}_{\aa,Y})$ which is the minimum of the numbers of nonzero components of nonzero vectors in $\cl C_{\aa,Y}$.

It is clear that the kernel of the linear map $\ev_{Y}$ equals the homogeneous piece $I(Y)_\aa$ of the vanishing ideal $I(Y)$ in degree $\aa$. Therefore, the code ${\cC}_{\aa,Y}$ is isomorphic to the $\K$-vector space $S_\aa/I(Y)_\aa$. Thus, the dimension of ${\cC}_{\aa,Y}$ is the multigraded Hilbert function $H_{Y}(\aa):=\dim_{\F_q} S_{\aa}-\dim_{\F_q} I(Y)_\aa$ of $I(Y)$. As ${\cC}_{\aa,Y}$ is a subspace of $\F_q^N$, we have $H_{Y}(\aa)\leq N$. When we have the equality $K=H_{Y}(\aa)= N$, the code becomes trivial, i.e. $\delta=\delta(\cl{C}_{\aa,Y})=1$, by the Singleton bound $\delta \leq N+1-K$. A related algebraic notion useful to eliminate these trivial codes is the so-called \textit{multigraded regularity} defined by $$\reg(Y):=\{\aa\in \N\bb \quad:\quad H_{Y}(\aa)= |Y|\}\subseteq \N^d.$$

The values of the Hilbert function $H_{Y}(\aa)$ can also be used to detect (monomially) equivalent codes. Indeed, by \cite[Proposition 4.3]{MultHFuncToricCodes16}, the codes ${\cC}_{\aa,Y}$ and ${\cC}_{\aa',Y}$ are equivalent if $H_{Y}(\aa)=H_{Y}(\aa')$ and $\aa-\aa'\in \N\bb$.

The first application of our results is that the length of the code ${\cC}_{\aa,Y}$ is given by $N=|Y|=c_1\cdots c_n$, due to Corollary \ref{c:c_i's}, where $Y=V_{X(\F_q)}(I_L)$ for the lattice $L$ with a basis given by $\{c_1\uu_1,\dots,c_n\uu_n\}$ for positive integers $c_i$ dividing $q-1$. 

In order to compute the dimension and minimum distance of the code ${\cC}_{\aa,Y}$, we need to determine a minimal generating set for the vanishing ideal $I(Y)$. We apply our Theorem \ref{t:nullstellensatz} together with a characterization of complete intersection lattice ideals given by Morales and Thoma \cite{MT2005} using the following concept.

\begin{defi} Let $A$ be a matrix whose entries are all integers. $A$ is called mixed if there is a positive and a negative entry in every column. If no square submatrix of $A$ is mixed, it is called dominating.
\end{defi} 

\begin{tm}\label{t:mixeddominating}\cite[Theorem 3.9]{MT2005} Let $L\subseteq\Z^r$ be a lattice with the property that $L\cap \N^r={0}$. Then, $I_L$ is complete intersection $\iff L$ has a basis ${\m_1,\dots,\m_k}$ such that the matrix $[\m_1\cdots \m_k]$ is mixed dominating. In the affirmative case, we have 
$ I_L=\la \textbf{x}^{\textbf{m}_1^+}-\textbf{x}^{\textbf{m}_1^-},\dots, \textbf{x}^{\textbf{m}_k^+}-\textbf{x}^{\textbf{m}_k^-} \ra. \hfill \Box$
\end{tm}

Applying our Theorem \ref{t:nullstellensatz} to $Y=V_{X(\F_q)}(I_L)$ we immediately obtain the equality of the ideals $I(Y)=I_L$. Using then Theorem \ref{t:mixeddominating}, one can confirm when $I(Y)=I_L$ is a complete intersection by looking at a basis of the lattice $L$. The rest of the section discusses an instance where everything works very well.

Let $X=\cl H_{\ell}$ be the Hirzebruch surface whose fan have primitive ray generators given by $\vv_1=(1,0)$, $\vv_2=(0,1)$, $\vv_3=(-1,\ell)$,
and $\vv_4=(0,-1)$, for any positive integer $\ell$. The exact sequence in (\ref{eq:SES}) becomes 
$$\dis \xymatrix{ \mathfrak{P}: 0  \ar[r] & \Z^2 \ar[r]^{\phi} & \Z^4 \ar[r]^{\beta}& \Cl(\cl H_{\ell}) \ar[r]& 0},$$
for $\phi=[\uu_1 \: \: \uu_2]$ with $\uu_1=(1,0,-1,0)$, $\uu_2=(0,1,\ell,-1)$ and $\beta=\begin{bmatrix}
1 & 0 & 1& \ell\\
0 & 1 & 0& 1  
\end{bmatrix}$ with $L_{\beta}=\la \uu_1, \uu_2\ra$.
The dual sequence in (\ref{eq:dualSES}) over $\K=\overline{\F}_q$ becomes
$$\dis \xymatrix{ \mathfrak{P}^*: 1  \ar[r] & G\ar[r]^{i} & (\K^*)^{4} \ar[r]^{\pi} & (\K^*)^2 \ar[r]& 1}$$
where $\pi:\t\mapsto (t_1t_3^{-1},t_2t_3^{\ell}t_4^{-1})$ and
$$G=\Ker(\pi)=\{(t_1,t_2,t_1,t_1^{\ell}t_2)\;|\;t_1,t_2\in\K^*\}\cong (\K^*)^2.$$ Hence, $\K$-rational points of the torus is $ T_X(\K) \cong {(\K^*)^2}\cong(\K^*)^{4} /G$ whereas $\F_q$-rational points is $ T_X(\F_q) \cong {(\F_q^*)^2}\cong(\F_q^*)^{4} /G$. 

The Cox ring $S=\F_q[x_1,x_2,x_3,x_4]$ is $\Z^2$-graded via  $$\deg(x_1)=\deg(x_3)=(1,0),\quad \deg(x_2)=(0,1), \quad \deg(x_4)=(\ell,1).$$ 

\begin{pro} Let $X$ be the Hirzebruch surface $ \mathcal{H_\ell}$. If $\{c_1\uu_1,c_2\uu_2\}$ constitute a basis for $L$, and $Y=V_{X(\F_q)}(I_L)$ then $I(Y)$ is a complete intersection generated minimally by $x_3^{c_1}-x_1^{c_1}$ and $x_4^{c_2}-x_2^{c_2}x_3^{\ell c_2}$. Furthermore, $(c_1+c_2\ell,c_2)+ \N^2 \subseteq \reg(Y)$. 
\
\end{pro}
\begin{proof} By Theorem \ref{t:nullstellensatz}, we get $I(Y)=I_L$.
Since the columns of $\verb|ML|$ constitute a basis of $L$, where
	$$\verb|ML|=[c_1 u_1 \quad c_2 u_2]=\begin{bmatrix}~c_1 & ~0 & -c_1&\quad 0\\~~~~~~0 & ~~ c_2 & \ell c_2 & ~ -c_2
	\end{bmatrix}^T$$ 
and $\verb|ML|$ is mixed dominating, it follows from Theorem \ref{t:mixeddominating} that $$I(Y)=\la x_3^{c_1}-x_1^{c_1},x_4^{c_2}-x_2^{c_2}x_3^{\ell c_2}\ra.$$
There is a geometrically significant subsemigroup $\cl K$ contained in the semigroup $\N\bb =\N^2$ whose elements are semiample degrees corresponding to the classes of numerically effective line bundles on $X$, which is given by 
$$\displaystyle \cl K =\bigcap_{\sigma \in \Sigma} \N \check{\sigma}=\N(1,0)+\N(\ell,1),$$ where $\N \check{\sigma}$ is the semigroup generated by $\bb_i$ corresponding to rays $\rho_i\notin \sigma$, for details see \cite[Theorem 6.3.12]{CLS-ToricVarieties}. Thus, the degrees $c_1(1,0)$ and $c_2(\ell,1)$ of the generators are semiample. Therefore, $(c_1+c_2\ell,c_2)+ \N^2 \subseteq \reg(Y)$ by \cite[Theorem 3.16]{MultHFuncToricCodes16}.
\end{proof}

 It is now time to the compute the other main parameters of the toric code $\mathcal{C}_{\aa,Y}$, where $\alpha=(c,d)\in \N\beta$. Hansen computed these parameters for the case $Y=T_X(\F_q)$, $c<q-1$ and $d=b$, where $b$ is to be defined below, see \cite{HansenAAECC}. More recently, the conditions $c<q-1$ and $d=b$ is relaxed in \cite{BaranSahin} again for the torus $Y=T_X(\F_q)$. 
 \begin{tm}\label{T:codesOnHirzebruch} Let $\{c_1\uu_1,c_2\uu_2\}$ be a basis for $L$ with $c_1$ and $c_2$ dividing $q-1$, and $Y=V_{X(\F_q)}(I_L)$ be the subgroup of the torus $T_X$ of the Hirzebruch surface $X= \mathcal{H_\ell}$ over $\K=\F_q$ satisfying $c_1\leq \ell c_2$. If $\alpha=(c,d)\in \N\beta=\N^2$, then the dimension $K$ of the toric code $\mathcal{C}_{\aa,Y}$  is given by

 \[
 \begin{cases}
(b+1)[c+1-b\ell /2], & \text{if }c<c_1 \text{ and } b\leq c_2-1 \\
  c_2[c+1- (c_2-1)\ell/2] & \text{if }c<c_1 \text{ and } c_2-1<b\\
  c_1(b'+1)+(b-b')[c+1-\ell(b+b'+1)/2], & \text{if }c\geq c_1\text{ and }b'\leq b < c_2-1\\
   c_1(b'+1)+(c_2-1-b')[c+1-\ell(c_2+b')/2], & \text{if }c\geq c_1 \text{ and } b'< c_2-1\leq b\\
   c_1c_2, & \text{if }c\geq c_1 \text{ and } b' \geq c_2-1,
 \end{cases}
 \] and  the minimum distance equals
 \[
 \delta({\cC}_{\aa,Y}) =
  \begin{cases}
 c_2[c_1-c], & \text{if }c<c_1 \text{ and } b\leq c_2-1 \\
  c_2[c_1-c], & \text{if }c<c_1 \text{ and }  c_2-1<b\\
  c_2-b', & \text{if }c\geq c_1\text{ and }b'\leq b < c_2-1\\
   c_2-b', & \text{if }c\geq c_1 \text{ and } b'< c_2-1\leq b\\
   1, & \text{if }c\geq c_1 \text{ and } b' \geq c_2-1.
 \end{cases}
 \] 
 where   $b$ is the greatest non-negative integer with $c-b\ell\geq 0\;\mbox{and}\;d-b\geq 0$, and when $c\geq c_1$, $b'$ is the greatest non-negative integer with  $c-b'\ell\geq c_1-1\;\mbox{and}\;d-b'\geq 0$. 
    \end{tm} 
   \begin{proof} %Consider  torus  $T_X$ of the Hirzebruch surface $ \mathcal{H_\ell}$ over $\K$ parameterized by identity matrix such that $\ell$ is positive integer. 
We start by giving a parameterization which will be crucial especially in computing the minimum distance below. So, we first see that  $$Q=\begin{bmatrix}~~(q-1)/c_1 & ~~0 & ~~0 &~~0\\~~0 & ~~0 & ~~0 & ~(q-1)/c_2 \end{bmatrix}$$  parameterizes the subgroup $Y=V_{X(\F_q)}(I_L)$. Let $\eta_1=\eta^{(q-1)/c_1}$ and $\eta_2=\eta^{(q-1)/c_2}$ where $\eta$ is a generator for the cyclic group $\K^*$. Then, $Y_Q$ consists of the following $c_1c_2$ points $P_{i,j}=[\eta_1^i:1:1:\eta_2^j]\in T_X$, for $i\in [c_1]$ and $j\in [c_2]$. Since these points satisfy the equations:
\[ x_1^{c_1}-x_3^{c_1}=0 \text{ and } x_4^{c_2}-x_2^{c_2}x_3^{\ell c_2}=0 \]
it follows that they belong to $Y$, yielding $Y_Q\subseteq Y$.	By Corollary \ref{c:c_i's}, $|Y|=c_1c_2$ forcing the equality $Y_Q=Y$.
	
Let us find  a $\K-$basis $B_\alpha:=\{\textbf{x}^\textbf{a}\:|\:\mbox{deg}(\textbf{x}^\textbf{a})=\beta \textbf{a}=\alpha\}$ for the vector space $S_\alpha$ for any $\alpha=(c,d)\in \N\beta=\N^2$. Since $b$ is the greatest non-negative integer with $\alpha=(c,d)=b(\ell,1)+(a,a')$ for some non-negative integers $a=c-b\ell\geq 0$ and $a'=d-b\geq 0$, we have $0\leq a_4\leq b$ if $\deg (\textbf{x}^\textbf{a})=\beta \textbf{a}=\alpha$. For a fixed  $0\leq a_4\leq b$, and fixed $0\leq a_1\leq c-\ell a_4$ the powers $a_2=d-a_4$ and $a_3=c-\ell a_4-a_1$ in $\textbf{x}^\textbf{a}=x_1^{a_1}x_2^{a_2}x_3^{a_3}x_4^{a_4}$ are fixed too. So, $$B_\alpha=\{x_1^{a_1}x_2^{d-a_4}x_3^{c-\ell a_4-a_1}x_4^{a_4}\:|\:\:0\leq a_4\leq b \text{ and } 0\leq a_1\leq c-\ell a_4\}.$$ 

Let us define the following two key numbers: 
\[ \mu_1:=\min \{c,c_1-1\} \text{ and } \mu_2:=\min \{b,c_2-1\}.
\]

Since $x_1^{c_1}= x_3^{c_1}$ and $x_4^{c_2}=x_2^{c_2}x_3^{\ell c_2}$ in the ring $S/I(Y)$, a basis for $S_{\aa}/I_{\aa}(Y)$ is  $$\bar{B}_\alpha=\{x_1^{a_1}x_2^{d-a_4}x_3^{c-\ell a_4-a_1}x_4^{a_4}\:|\:\:0\leq a_4\leq \mu_2 \text{ and } 0\leq a_1\leq \min \{c-\ell a_4,\mu_1\}\}.$$
It is clear that $S_{(c,d)}=x_2^{d-b}S_{(c,b)}$. Since $x_2=1$ on $Y$, the images $\ev_Y(S_{(c,d)})$ and $\ev_Y(S_{(c,b)})$ of the evaluation maps defined on $S_{(c,d)}$ and $S_{(c,b)}$ are the same code ${\cC}_{(c,b),Y}$ when $d>b$. Hence, if $\lfloor c /\ell \rfloor < d$, we have $ b=\min \{\lfloor c /\ell \rfloor,d\}=\lfloor c /\ell \rfloor$, and thus the codes ${\cC}_{(c,d),Y}$ and ${\cC}_{(c,b),Y}$ are the same whose dimensions are
\[\mbox{dim}_{\F_q}{\cC}_{\aa,Y}=H_Y(\aa)=|\bar{B}_{(c,b)}|, \text{ for } \aa=(c,d) \text{ with } b=\min \{\lfloor c /\ell \rfloor,d\}=\lfloor c /\ell \rfloor.
\]
Therefore, it suffices to study codes ${\cC}_{\aa,Y}$, where $\aa=(c,d)$ with $d\leq \lfloor c /\ell \rfloor$. Thus, we assume from now on that 
\begin{equation}
 \label{eq:assumptionB}
d\leq \lfloor c /\ell \rfloor \text{ so that } b=\min \{\lfloor c /\ell \rfloor,d\}=d.   
\end{equation} 

\textbf{The dimension.}

\textbf{Case I:} Let $c=a+b\ell<c_1$. Then, $\mu_1=c$ and thus we have
\begin{eqnarray} \label{e:dimensionCase1}
\quad \mbox{dim}_{\F_q}{\cC}_{\aa,Y}&=&|\bar{B}_\alpha|=\sum\limits_{a_4=0}^{\mu_2}(c-\ell a_4+1) \\
&=& (\mu_2+1)(c+1)- \mu_2(\mu_2+1)\ell/2=(\mu_2+1)[c+1- \mu_2 \ell/2]. \nonumber
\end{eqnarray}

\begin{figure}[htb!]
\centering
\begin{minipage}{0.5\textwidth}
  \centering
\begin{tikzpicture}[scale=0.7]
\def\q{7};
\def\cOne{\q};
\def\cTwo{\q-2};
\def\l{2};

\def\bPrime{3};
\pgfmathsetmacro\c{\q-2};
\pgfmathsetmacro\cl{0.5*(\c)};
  %\pgfmathsetmacro\b'{Floor(3.5)};
			
			\pgfmathsetmacro\qMinusOne{\q-1};
			\pgfmathsetmacro\d{2};
            \pgfmathsetmacro\b{2};

\draw[step=1cm,gray!30!white,very thin] (0,0) grid (\q,\q-1);
\draw[thick,->] (0,0) -- (\q,0) node[anchor=south east] {};
\draw[thick,->] (0,0) -- (0,\q-1) node[anchor=south east] {};
\draw (\c,0) -- (0,\cl);

%\draw [dashed, blue, ultra thick] (1.75,4) -- (0,3.5);
\draw[dashed, red , thick] (\cOne-1,0) -- (\cOne-1,\cTwo-1) --(0,\cTwo-1);
\node [left] at (0,\cTwo-1) {$c_2$-$1$};
\node [below] at (\cOne-1,0) {$c_1$-$1$};

\draw[dashed, red , thick] (\c-\l*\b,0) -- (\c-\l*\b,\b) --(0,\b);
%\draw[dashed, red , thick] (\cOne-1,\cl-1) --(0,\cl-1);
%\draw[dashed, red , thick] (\c-\l*\bPrime,0) -- (\c-\l*\bPrime,\bPrime) --(0,\bPrime);
\node [above right] at (\c-\l*\b,\b) {$a_1+\ell a_4=c$};
\node [left] at (-0.12,\cl+0.14) {$c/\ell$};
%\node [left] at (0,\cl-1) {$(c$-$c_1$+$1$)/\ell$};

%\node [left] at (0,\bPrime) {$b'$};
\node [left] at (0,\b) {$b=d$};
\node [below] at (\c,-0.12) {$c$};
\node [below left] at (0,0) {$0$};
%\node [below] at (\c-\l*\bPrime,0) {$c-\ell b'$};
\node [below] at (\c-\l*\b,0) {$c-\ell b$};
\filldraw[fill=blue!40!white, draw=black] (0,0) rectangle (\c-\l*\b,\b);
\filldraw[fill=red!40!white, draw=black] (\c-\l*\b,0)-- (\c-\l*\b,\b)--(\c,0)--cycle;
\end{tikzpicture} 
  \caption{Case I with $b\leq c_2-1$}
  \label{fig:case1a}
\end{minipage}%
\begin{minipage}{0.5\textwidth}
  \centering
 \begin{tikzpicture}[scale=0.7]
\def\q{7};
\def\cOne{\q};
\def\cTwo{2};
\def\l{2};

\def\bPrime{3};
\pgfmathsetmacro\c{\q-2};
\pgfmathsetmacro\cl{0.5*(\c)};
  %\pgfmathsetmacro\b'{Floor(3.5)};
			
			\pgfmathsetmacro\qMinusOne{\q-1};
			\pgfmathsetmacro\d{2};
            \pgfmathsetmacro\b{2};

\draw[step=1cm,gray!30!white,very thin] (0,0) grid (\q,\q-1);
\draw[thick,->] (0,0) -- (\q,0) node[anchor=south east] {};
\draw[thick,->] (0,0) -- (0,\q-1) node[anchor=south east] {};
\draw (\c,0) -- (0,\cl);

%\draw [dashed, blue, ultra thick] (1.75,4) -- (0,3.5);
\draw[dashed, red , thick] (\cOne-1,0) -- (\cOne-1,\cTwo-1) --(0,\cTwo-1);
\node [left] at (0,\cTwo-1) {$c_2$-$1$};
\node [below] at (\cOne-1,0) {$c_1$-$1$};

\draw[dashed, red , thick] (\c-\l*\b,0) -- (\c-\l*\b,\b) --(0,\b);
%\draw[dashed, red , thick] (\cOne-1,\cl-1) --(0,\cl-1);
%\draw[dashed, red , thick] (\c-\l*\bPrime,0) -- (\c-\l*\bPrime,\bPrime) --(0,\bPrime);
\node [above right] at (\c-\l*\b,\b) {$a_1+\ell a_4=c$};
\node [left] at (-0.12,\cl+0.14) {$c/\ell$};
%\node [left] at (0,\cl-1) {$(c$-$c_1$+$1$)/\ell$};

%\node [left] at (0,\bPrime) {$b'$};
\node [left] at (0,\b) {$b=d$};
\node [below] at (\c,-0.12) {$c$};
\node [below left] at (0,0) {$0$};
%\node [below] at (\c-\l*\bPrime,0) {$c-\ell b'$};
\node [below] at (\c-\l*\b,0) {$c-\ell b$};
\filldraw[fill=blue!40!white, draw=black] (0,0) rectangle (\c-\l*\cTwo+\l,\cTwo-1);
\filldraw[fill=red!40!white, draw=black] (\c-\l*\cTwo+\l,0)-- (\c-\l*\cTwo+\l,\cTwo-1)--(\c,0)--cycle;
\end{tikzpicture} 
  \caption{Case I with $b > c_2-1$}
  \label{fig:case1b}
\end{minipage}
\end{figure}

If $b\leq c_2-1$, then $\mu_2=b$ and thus, we have 
\begin{eqnarray} 
\dim_{\F_q}{\cC}_{\aa,Y}= (b+1)[c+1- b\ell/2]. \nonumber
\end{eqnarray}
It is worth to notice that this dimension is the number of monomials in 
\[
\bar{B}_\alpha=\{x_1^{a_1}x_2^{d-a_4}x_3^{c-\ell a_4-a_1}x_4^{a_4}\:|\:\:0\leq a_4\leq b \text{ and } 0\leq a_1\leq c-\ell a_4\}
\]
or equivalently the number of lattice points inside the polygon depicted in Figure \ref{fig:case1a} which is defined by the inequalities $0\leq a_4\leq b \text{ and } 0\leq a_1\leq c-\ell a_4$.

If $c_2-1<b$, then $\mu_2=c_2-1$ and hence, we have 
\begin{eqnarray} 
\dim_{\F_q}{\cC}_{\aa,Y}= c_2[c+1- (c_2-1)\ell/2]. \nonumber
\end{eqnarray}

As before, this dimension is the number of monomials in 
\[
\bar{B}_\alpha=\{x_1^{a_1}x_2^{d-a_4}x_3^{c-\ell a_4-a_1}x_4^{a_4}\:|\:\:0\leq a_4\leq c_2-1 \text{ and } 0\leq a_1\leq c-\ell a_4\}
\]
or equivalently the number of lattice points inside the polygon depicted in Figure \ref{fig:case1b} which is defined by the inequalities $0\leq a_4\leq c_2-1 \text{ and } 0\leq a_1\leq c-\ell a_4$.

 \textbf{Case II:} Let $c\geq c_1 $. Then $\mu_1=c_1-1$. Recall that $b'$ is the greatest non-negative integer with the property  $c-b'\ell\geq c_1-1\;\mbox{and}\;d-b'\geq 0$. In fact, if $a\geq c_1-1$, then $c-\ell b=a \geq c_1-1$, and thus $b'=b$. Otherwise, $$b'=\lfloor (c-c_1+1)/\ell \rfloor=b-\lceil (c_1-1-a)/\ell \rceil.$$ 
 
 \textbf{Case II(i): $b'\geq c_2-1$.} 
 \begin{figure}[htb!]
     \centering
     \begin{tikzpicture}[scale=0.8]
\def\q{11};
\def\cOne{\q-7};
\def\cTwo{3};
\def\l{2};

\def\bPrime{3};
\pgfmathsetmacro\c{\q-1};
\pgfmathsetmacro\cl{0.5*(\c)};
  %\pgfmathsetmacro\b'{Floor(3.5)};
			
			\pgfmathsetmacro\qMinusOne{\q-1};
			\pgfmathsetmacro\d{3};
            \pgfmathsetmacro\b{3};

\draw[step=1cm,gray!30!white,very thin] (0,0) grid (\q,\q-4);
\draw[thick,->] (0,0) -- (\q,0) node[anchor=south east] {};
\draw[thick,->] (0,0) -- (0,\q-4) node[anchor=south east] {};
\draw (\c,0) -- (0,\cl);

%\draw [dashed, blue, ultra thick] (1.75,4) -- (0,3.5);
\draw[dashed, red , thick] (\cOne-1,0) -- (\cOne-1,\cTwo-1) --(0,\cTwo-1);
\node [left] at (0,\cTwo-1) {$c_2$-$1$};
\node [below] at (\cOne-1,0) {$c_1$-$1$};

\draw[dashed, red , thick] (\c-\l*\b,0) -- (\c-\l*\b,\b) --(0,\b);
%\draw[dashed, red , thick] (\cOne-1,\cl-1) --(0,\cl-1);
%\draw[dashed, red , thick] (\c-\l*\bPrime,0) -- (\c-\l*\bPrime,\bPrime) --(0,\bPrime);
\node [above right] at (\c-\l*\b,\b) {$a_1+\ell a_4=c$};
\node [left] at (0,\cl) {$c/\ell$};
%\node [left] at (0,\cl-1) {$(c$-$c_1$+$1$)/\ell$};

%\node [left] at (0,\bPrime) {$b'$};
\node [left] at (0,\b) {$b'$};
\node [below] at (\c,0) {$c$};
\node [below left] at (0,0) {$0$};
%\node [below] at (\c-\l*\bPrime,0) {$c-\ell b'$};
\node [below] at (\c-\l*\b,0) {$c$-$\ell b'$};
\filldraw[fill=blue!40!white, draw=black] (0,0) rectangle (\cOne-1,\cTwo-1);

\end{tikzpicture}  
     \caption{Case II (i)}
     \label{fig:case2i}
 \end{figure}
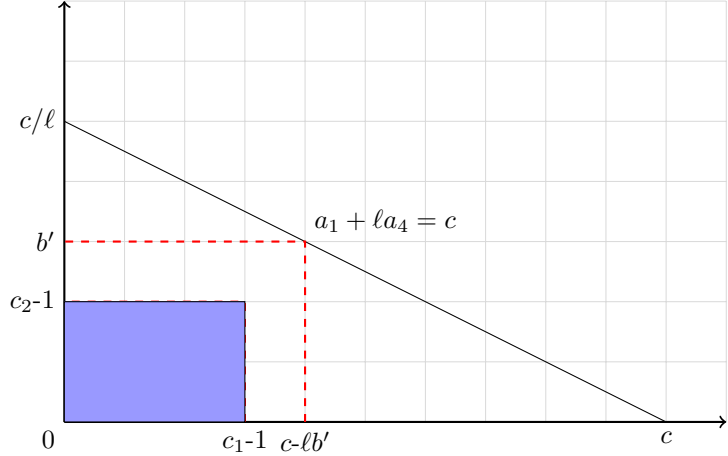

 If $b'\geq c_2-1$, then $\mu_2=c_2-1$ as $b\geq b'$. So, for every choice of $a_4$ satisfying $0 \leq a_4 \leq c_2-1\leq b'$ we have $c-\ell a_4 \geq c_1-1$, and hence $0\leq a_1 \leq c_1-1$. Thus, we have 
\begin{eqnarray} \label{e:dimensionN}
\mbox{dim}_{\F_q}{\cC}_{\aa,Y}=|\bar{B}_\alpha|=c_1c_2=N,
\end{eqnarray}
which is also the number of lattice points in the rectangle depicted in Figure \ref{fig:case2i} and described by the inequalities $0\leq a_1 \leq c_1-1$ and $0 \leq a_4 \leq c_2-1$. Since the dimension reaches its maximum value, the code ${\cC}_{\aa,Y}$ is trivial, that is, $\delta({\cC}_{\aa,Y})=1$. 

 \textbf{Case II(ii): $b' < c_2-1$.} 
 
By the definition of $b'$, we have the following two cases
\[ \min \{c-\ell a_4,c_1-1\}= \begin{cases} 
      c_1-1 & \text{ if } \quad 0\leq a_4 \leq b' \\
      c-\ell a_4 & \text{ if } \quad b'<a_4\leq \mu_2.
   \end{cases}
\]
 So, if $0\leq a_4\leq b'$ then $0\leq a_1 \leq c_1-1$ but if $a_4 > b'$ then $0\leq a_1 \leq c-\ell a_4$, yielding the description of $\bar{B}_\alpha$ to be the disjoint union of the following two sets
\begin{eqnarray} \label{e:decompositionBalpha}
\quad \quad \bar{B}_\alpha(1)&=&\{x_1^{a_1}x_2^{d-a_4}x_3^{c-\ell a_4-a_1}x_4^{a_4}\:|\:\:0\leq a_4\leq b' \text{ and } 0\leq a_1\leq c_1-1\} \text{ and } \\
\quad \quad \bar{B}_\alpha(2)&=&\{x_1^{a_1}x_2^{d-a_4}x_3^{c-\ell a_4-a_1}x_4^{a_4}\:|\:\:b'+1\leq a_4\leq \mu_2 \text{ and } 0\leq a_1\leq c-\ell a_4\}.
\end{eqnarray}
Notice that $|\bar{B}_\alpha(1)|$ is the number of lattice points in the rectangle in Figure \ref{fig:case2ii} and $|\bar{B}_\alpha(2)|$ is the number of lattice points in the trapezoid in Figure \ref{fig:case2ii}.
Therefore, we have
\begin{eqnarray} \label{e:dimensionFormula}
\mbox{dim}_{\F_q}{\cC}_{\aa,Y}&=&|\bar{B}_\alpha|=c_1(b'+1)+\sum\limits_{a_4=b'+1}^{\mu_2}(c-\ell a_4+1)  \\
&=& c_1(b'+1)+ (\mu_2-b')[c+1-(\mu_2+b'+1)\ell/2]. \nonumber
\end{eqnarray}

 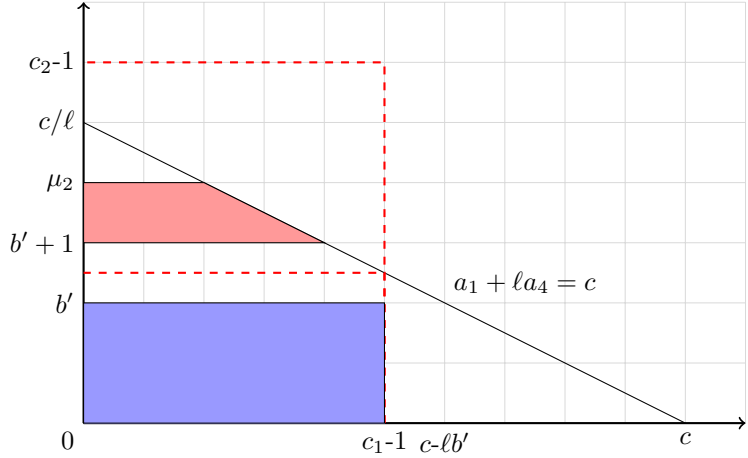
\begin{figure}[htb!]
     \centering
     \begin{tikzpicture}[scale=0.8]
\def\q{11};
\def\cOne{\q-5};
\def\cTwo{7};
\def\l{2};

\def\bPrime{2};
\pgfmathsetmacro\c{\q-1};
\pgfmathsetmacro\cl{0.5*(\c)};
  %\pgfmathsetmacro\b'{Floor(3.5)};
			
			\pgfmathsetmacro\qMinusOne{\q-1};
			\pgfmathsetmacro\d{3};
            \pgfmathsetmacro\b{3};

\draw[step=1cm,gray!30!white,very thin] (0,0) grid (\q,\q-4);
\draw[thick,->] (0,0) -- (\q,0) node[anchor=south east] {};
\draw[thick,->] (0,0) -- (0,\q-4) node[anchor=south east] {};
\draw (\c,0) -- (0,\cl);

\draw [dashed, red, thick] (\cOne-1,0) -- (\cOne-1,0.5+\bPrime) --(0,0.5+\bPrime);
\draw[dashed, red , thick] (\cOne-1,0) -- (\cOne-1,\cTwo-1) --(0,\cTwo-1);
\node [left] at (0,\cTwo-1) {$c_2$-$1$};
\node [below] at (\cOne-1,0) {$c_1$-$1$};

%\draw[dashed, red , thick] (\cOne-1,\cl-1) --(0,\cl-1);
%\draw[dashed, red , thick] (\c-\l*\bPrime,0) -- (\c-\l*\bPrime,\bPrime) --(0,\bPrime);
\node [above right] at (\c-\l*\bPrime,\bPrime) {$a_1+\ell a_4=c$};
\node [left] at (0,\cl) {$c/\ell$};
%\node [left] at (0,\cl-1) {$(c$-$c_1$+$1$)/\ell$};
\node [left] at (0,\bPrime+2) {$\mu_2$};
\node [left] at (0,\bPrime+1) {$b'+1$};
\node [left] at (0,\bPrime) {$b'$};
\node [below] at (\c,0) {$c$};
\node [below left] at (0,0) {$0$};
%\node [below] at (\c-\l*\bPrime,0) {$c-\ell b'$};
\node [below] at (\c-\l*\bPrime,0) {$c$-$\ell b'$};
\filldraw[fill=blue!40!white, draw=black] (0,0) rectangle (\cOne-1,\bPrime);
\filldraw[fill=red!40!white, draw=black] (0,1+\bPrime)--(\c-\l*\bPrime-\l,\bPrime+1)--(\c-\l*\bPrime-\l-2,\bPrime+2)-- (0,\bPrime+2)--cycle;
\end{tikzpicture}  
     \caption{Case II (ii)}
     \label{fig:case2ii}
 \end{figure}

 If $b < c_2-1$, then $\mu_1=c_1-1$ and $\mu_2=b$ in (\ref{e:dimensionFormula}), yielding the formula 
\begin{eqnarray} 
\dim_{\F_q}{\cC}_{\aa,Y}= c_1(b'+1)+ (b-b')[c+1-(b+b'+1)\ell/2]. \nonumber
\end{eqnarray}
 
If $b'< c_2-1 \leq b$, then $\mu_1=c_1-1$ and $\mu_2=c_2-1$ in (\ref{e:dimensionFormula}), yielding the formula 
\begin{eqnarray} 
\dim_{\F_q}{\cC}_{\aa,Y}= c_1(b'+1)+ (c_2-1-b')[c+1-(c_2+b')\ell/2]. \nonumber
\end{eqnarray}

 \textbf{The minimum distance.}
 
Let us compute the minimum distance now. Recall that $Y=V_{X(\F_q)}(I_L)$ consists of the following points $P_{i,j}=[\eta_1^i:1:1:\eta_2^j]\in T_X$, for $i\in [c_1]$ and $j\in [c_2]$.
  
For a given $F\in S_{\aa}$ we set $J=\{j\in [c_2] : x_4-\eta_2^j x_2x_3^{\ell} \quad \mbox{divides} \quad F\}$. Then, we first prove that the number of zeroes of $F$ in $Y$ satisfies  
\begin{eqnarray} \label{e:V_Y(F)first}
|V_X(F)\cap Y| \leq c_1|J|+(c_2-|J|) \deg_{x_1} F. 
\end{eqnarray} 
 It is clear that $F(P_{i,j})=0$ whenever $j\in J$ and there are $c_1|J|$ many such points in $Y$. Moreover, the polynomial $f(x_1):=F(x_1,1,1,\eta_2^j)\in \K[x_1]\setminus \{0\}$ if $j\notin J$. Thus, it can have at most $\deg_{x_1} F:=\deg (f)$ many zeroes for each $j\notin J$. So, we can have at most $(c_2-|J|) \deg_{x_1} F$ many such roots. Therefore, altogether $F$ can have at most $ c_1|J|+(c_2-|J|) \deg_{x_1} F$ zeroes in $Y$.

  \textbf{Case I:} Let $c<c_1$. \\
 Since $\deg_{x_1} F\leq c-\ell |J|$ in this case, it follows from (\ref{e:V_Y(F)first}) that
\begin{eqnarray*} \label{e:V_Y(F)firstcase}
|V_X(F)\cap Y| &\leq& c_1|J|+(c_2-|J|) (c-\ell |J|)\\
&\leq &c_2c+|J|(c_1-\ell c_2-c+\ell|J|). 
\end{eqnarray*} 
Since $0\leq |J| \leq \lfloor c/ \ell \rfloor$ and $c_1\leq \ell c_2$ , we have
 \begin{eqnarray} \label{e:V_Y(F)case1}
 |V_X(F)\cap Y| &\leq& c_2c+|J|(c_1-\ell c_2-c+\ell|J|)\leq c_2c.
\end{eqnarray}
Therefore, $F$ can have at most $c_2c$ many zeroes in $Y$. On the other hand, the following polynomial  
   $$F=x_2^
    d\prod\limits_{i=1}^{c}(x_1-\eta_1^ix_3)\in S_\alpha$$ vanishes exactly at the $c_2c$ points $P_{i,j}=[\eta_1^i:1:1:\eta_2^j]\in Y$, where $1\leq i \leq c$ and $1\leq j\leq c_2$. Thus,  there is a codeword $\mbox{ev}_{\alpha,Y}(F)$ with weight $c_1c_2-c_2c$. Hence,  $$\delta({\cC}_{\aa,Y})=c_1c_2-c_2c=c_2(c_1-c).$$ 
    
   \textbf{Case II:} Let $c\geq c_1$. 
 Then, we elaborate more on the upper bound given in (\ref{e:V_Y(F)first}) depending on whether $|J|\leq b'$ or $|J|>b'$.
 
 Suppose that $0\leq |J|\leq b'$. Since $\deg_{x_1} F \leq c_1-1$, it follows from (\ref{e:V_Y(F)first}) that
 \begin{eqnarray} 
|V_X(F)\cap Y| &\leq& c_1|J|+(c_2-|J|) (c_1-1)= c_2 (c_1-1)+|J| \nonumber. 
\end{eqnarray}
Since  $0\leq |J|\leq b'$, we have 
 \begin{eqnarray}  \label{e:V_Y(F)}
  |V_X(F)\cap Y|  \leq
c_2(c_1-1)+b'. 
\end{eqnarray}

Suppose now that $|J|>b'$. In this case, we write $|J|=b'+k$ for some positive integer $k$. Then, $c-\ell (b'+1) < c_1-1$ or equivalently $c-\ell (b'+1) \leq c_1-2 $, by definition of $b'$, which will be crucial below. Hence, 
\[c-\ell |J|=c-\ell(b'+1)- \ell(|J|-b'-1)\leq c_1-2-\ell(k-1).
\] Since $\deg_{x_1} F \leq c-\ell |J|\leq c_1-2-\ell(k-1)$, it follows from (\ref{e:V_Y(F)first}) that
   \begin{eqnarray} 
|V_X(F)\cap Y| &\leq& c_1|J|+(c_2-|J|) (c-\ell |J|-c_1+1+c_1-1) \nonumber\\
&\leq& c_1c_2-c_2+(c_2-|J|) (c-\ell |J|-c_1+1)+|J| \nonumber\\
&\leq& c_2(c_1-1)+(c_2-|J|) (-1-\ell(k-1))+|J|\nonumber.
\end{eqnarray}
Since $b'-|J|=-k$, $c_2-|J| \geq 1$ and  $\ell \geq 1$ the difference 
$$b'-[|J|+(c_2-|J|) (-1-\ell(k-1))]=(b'-|J|) +(c_2-|J|)(1+\ell(k-1))\geq 0$$  
and thus we have that
\begin{eqnarray} \label{e:V_Y(F)2}
|V_X(F)\cap Y| \leq (c_1-1)c_2+b'.
\end{eqnarray}

Then, by (\ref{e:V_Y(F)}) and (\ref{e:V_Y(F)2}), $F$ can have at most $(c_1-1)c_2+b'$ many zeroes in $Y$. On the other hand, the following polynomial  
   $$F=x_3^{c-\ell b'-(c_1-1)}x_2^{d-b'}\prod\limits_{i=1}^{c_1-1}(x_1-\eta_1^ix_3)\prod\limits_{j=1}^{b'}(x_4-\eta_2^j x_2x_3^{\ell})$$
vanishes exactly at $c_1b'+(c_2-b')(c_1-1)=c_2(c_1-1)+b'$, and thus there is a codeword with  weight $c_2-b'$. This shows that
	$\delta({\cC}_{\aa,Y})=c_2-b'$. 
\end{proof}
\begin{rema} \label{r:a<c_1-1}
    Let $(c,d)=(a+\ell b, a'+b)$. If $a\geq c_1-1$, then $c-\ell b=a\geq c_1-1$. So, $b'=b$. Since $b'=b < c_2-1$, then by (\ref{e:dimensionFormula}) above  $H_Y(\aa)=|\bar{B}_{\aa}|=c_1(b+1)$. It follows that the codes ${\cC}_{\aa',Y}$ with $a>c_1-1$ are all equivalent to the code ${\cC}_{\aa,Y}$ with $a=c_1-1$, by \cite[Proposition 4.3]{MultHFuncToricCodes16}. Furthermore, as the codes with $a'>0$ are equivalent to those with $a'=0$ by the explanations above (\ref{eq:assumptionB}), and the codes with $b'\geq c_2-1$ are trivial, it suffices to consider codes corresponding to $(c,d)=(a+\ell b, b)$ where $a\leq c_1-1$ and $b'<c_2-1$.
\end{rema}
The following makes the points of Remark \ref{r:a<c_1-1} more precise as the extra condition $\ell \geq c_1$ bounds $b$ by $c_2-1$ and reduces the number of non-equivalent codes to $c_1c_2$. 
\begin{coro} Let $\{c_1\uu_1,c_2\uu_2\}$ be a basis for $L$, and $Y=V_X(I_L)(\F_q)$ be the subgroup of the torus $T_X$ of the Hirzebruch surface $ \mathcal{H_\ell}$ over $\K=\F_q$ for which $c_1$ and $c_2$ divides $q-1$. If $\aa=(c,d)=(a+\ell b,d)$ and $\ell \geq c_1$, then the parameters of the toric code $\mathcal{C}_{\aa,Y}$  are $N=c_1c_2$,

$$\begin{array}{lllllllll} 
&K&=&a+1 &\text{ and }      &\delta &=& c_2(c_1-a), &\text{ if } c<c_1\\
&K&=&c_1b+1+a &\text{ and } &\delta &=& c_2-b+1, &\text{ if } a<c_1-1<c\\
&K&=&c_1(b+1) &\text{ and } &\delta &=& c_2-b, &\text{ if } a=c_1-1<c.
\end{array}$$

Furthermore, $(c_1-1+\ell(c_2-1),c_2-1)+ \N^2 = \reg(Y)$.
\end{coro}
\begin{proof}
Let $\ell \geq c_1$. We may suppose that $a\leq c_1-1$ by the virtue of Remark \ref{r:a<c_1-1}. Since $a\leq c_1-1$, it follows that $b'=b$ if $a=c_1-1$ and $b'=b-1$ if $a=c-\ell b<c_1-1$ as $c-\ell (b-1)=c-\ell b+ \ell \geq c_1-1$. So, if $b\geq c_2$, then $b'\geq c_2-1$ and thus the code is trivial by (\ref{e:dimensionN}). Thus, non-equivalent codes occur only for $\aa=a(1,0)+b(\ell,1)$ where 
$0\leq a \leq c_1-1$ and  $0\leq b \leq c_2-1$.

 \textbf{Case I:} If $c=a+\ell b<c_1\leq \ell$, then $0\leq a < \ell(1-b)$ implying that $b=0$. Hence, $\dim_{\K}{\cC}_{\aa,Y}=H_Y(a,0)=a+1$ and $\delta({\cC}_{\aa,Y})=c_2(c_1-a)$ by Theorem \ref{T:codesOnHirzebruch}, since $c=a$.
 
  \textbf{Case II:} Let $c\geq c_1$. We appeal to Theorem \ref{T:codesOnHirzebruch} again.
  
  If $a=c_1-1$, then $b'=b$ and $\dim_{\K}{\cC}_{\aa,Y}=H_Y(\aa)=c_1(b+1)$. Furthermore, we have $\delta({\cC}_{\aa,Y})=c_2-b'=c_2-b$. 
  
  If $a<c_1-1$, then $b'=b-1$ and $\dim_{\K}{\cC}_{\aa,Y}=H_Y(\aa)=c_1b+c+1-\ell b=c_1b+1+a$. Moreover, we have $\delta({\cC}_{\aa,Y})=c_2-b'=c_2-b+1$.
  
  The final claim follows easily, since the there values $a+1, c_1b+1+a$ and $ c_1(b+1)$ of $K$ above are strictly smaller than $c_1c_2$ whenever $0\leq a < c_1-1$ and  $0\leq b < c_2-1$, and $K=H_Y((c_1-1+\ell(c_2-1),c_2-1))=c_1c_2$ as $a = c_1-1$ and  $ b = c_2-1$. 
 \end{proof}
We finish the paper with a toy example illustrating the theory we developed.
\begin{ex} Let $X=\cl H_3$ over $\F_{7}$. Let $c_1=3$, $c_2=3$. Then the parameters of the code are given in Table \ref{table:codes}. According to Markus Grassl's Code Tables \cite{Grassl:codetables} a best-possible code with $N=9$ has $K+\delta=9$ or $K+\delta=10$ (MDS codes). This example provides us with a best possible code whose parameters are $[9,7,2]$ together with an MDS code $[9,8,2]$.

%\begin{table}[h!]--Tabloyu metinde nereye koyarsan orada görünür
%  \hline
% \multicolumn{5}{|c|c|c|}{Codes on $Y_Q \subseteq \cl H_2$}& \multicolumn{5}{| c |}{Codes on $T_X \subseteq X=\cl H_2$}\\

\begin{longtable}[h!]{|c|c|c|} 
\caption{Parameters of Toric Codes on $X=\cl H_3$. \label{table:codes}}\\	
	
	\hline
	{$(a, b)$ \cellcolor{gray!20}} &
	{$\mathbf{\alpha}$ \cellcolor{gray!20}}& 
	{$[N, K, \delta]$ \cellcolor{gray!20}}\\
	
	\hline
$(0, 0)$ & $(0, 0)$ & $[9, 1, 9]$  \\
\hline
$(1, 0)$ & $ (1, 0)$ & $[9, 2, 6]$ \\ 
\hline
$(2, 0)$ &  $(2, 0)$ & $[9, 3, 3] $ \\ 
	\hline
$(0, 1)$ & $(3, 1)$ &$[9, 4, 3]$  \\
\hline
$(1, 1)$ &  $(4, 1)$ & $[9, 5, 3]$  \\
\hline
$(2, 1)$ &  $(5, 1)$ & $[9, 6, 2]$ \\ 
	\hline
$(0, 2)$ &  $(6, 2)$ &$[9, 7, 2]$  \\
\hline
$(1, 2)$ &  $(7, 2)$ & $[9, 8, 2]$\\
\hline
$(2, 2)$ &  $(8, 2)$ &   $[9, 9, 1] $  \\
	\hline
\end{longtable}	
\end{ex}

\section*{Acknowledgement} The author thanks Laurence Barker for his contribution explained in Remark \ref{r:barker} and Ivan Soprunov for his useful comments on the manuscript. He also expresses his gratitude to an anonymous referee who reads the paper carefully and helped improve the presentation considerably.

%\bibliographystyle{amsplain}
%\bibliographystyle{plain} 
%\bibliography{SahinRationalPointsOfLatticeIdealsOnAToricVariety} 

\begin{thebibliography}{10}

\bibitem{BaranSahin}
Esma Baran and Mesut \c{S}ahin.
\newblock On parameterised toric codes.
\newblock {\em Appl. Algebra Engrg. Comm. Comput.}, 2021.

\bibitem{EsmaTJM22}
Esma Baran~\"{O}zkan.
\newblock Vanishing ideals of parameterized subgroups in a toric variety.
\newblock {\em Turkish J. Math.}, 46(6):2141--2150, 2022.

\bibitem{OrderBound}
Peter Beelen and Diego Ruano.
\newblock The order bound for toric codes.
\newblock In {\em Applied algebra, algebraic algorithms, and error-correcting
  codes}, volume 5527 of {\em Lecture Notes in Comput. Sci.}, pages 1--10.
  Springer, Berlin, 2009.

\bibitem{Bernstein}
D.~N. Bernstein.
\newblock The number of roots of a system of equations.
\newblock {\em Funkcional. Anal. i Prilo\v{z}en.}, 9(3):1--4, 1975.

\bibitem{7championCodes}
Gavin Brown and Alexander~M. Kasprzyk.
\newblock Seven new champion linear codes.
\newblock {\em LMS J. Comput. Math.}, 16:109--117, 2013.

\bibitem{CoxRingToric}
David~A. Cox.
\newblock The homogeneous coordinate ring of a toric variety.
\newblock {\em J. Algebraic Geom.}, 4(1):17--50, 1995.

\bibitem{CLS-ToricVarieties}
David~A. Cox, John~B. Little, and Henry~K. Schenck.
\newblock {\em Toric varieties}, volume 124 of {\em Graduate Studies in
  Mathematics}.
\newblock American Mathematical Society, Providence, RI, 2011.

\bibitem{CodesWPS}
Eduardo Dias and Jorge Neves.
\newblock Codes over a weighted torus.
\newblock {\em Finite Fields Appl.}, 33:66--79, 2015.

\bibitem{BinomialIdeals}
David Eisenbud and Bernd Sturmfels.
\newblock Binomial ideals.
\newblock {\em Duke Math. J.}, 84(1):1--45, 1996.

\bibitem{Ghorpade19}
Sudhir~R. Ghorpade.
\newblock A note on {N}ullstellensatz over finite fields.
\newblock In {\em Contributions in algebra and algebraic geometry}, volume 738
  of {\em Contemp. Math.}, pages 23--32. Amer. Math. Soc., [Providence], RI,
  [2019] \copyright 2019.

\bibitem{Grassl:codetables}
Markus Grassl.
\newblock {Bounds on the minimum distance of linear codes and quantum codes}.
\newblock Online available at http://www.codetables.de, 2007.
\newblock Accessed on 2023-01-17.

\bibitem{HansenAAECC}
Johan~P. Hansen.
\newblock Toric varieties {H}irzebruch surfaces and error-correcting codes.
\newblock {\em Appl. Algebra Engrg. Comm. Comput.}, 13(4):289--300, 2002.

\bibitem{HuStu95}
Birkett Huber and Bernd Sturmfels.
\newblock A polyhedral method for solving sparse polynomial systems.
\newblock {\em Mathematics of Computation}, 64(212):1541--1555, 1995.

\bibitem{Jacobson}
Nathan Jacobson.
\newblock {\em Basic algebra. {I}}.
\newblock W. H. Freeman and Company, New York, second edition, 1985.

\bibitem{Jaramillo21DCC}
Delio Jaramillo, Maria Vaz~Pinto, and Rafael~H. Villarreal.
\newblock Evaluation codes and their basic parameters.
\newblock {\em Des. Codes Cryptogr.}, 89(2):269--300, 2021.

\bibitem{JoynerAAECC}
David Joyner.
\newblock Toric codes over finite fields.
\newblock {\em Appl. Algebra Engrg. Comm. Comput.}, 15(1):63--79, 2004.

\bibitem{Kushnirenko}
A.~G. Ku\v{s}nirenko.
\newblock Newton polyhedra and {B}ezout's theorem.
\newblock {\em Funkcional. Anal. i Prilo\v{z}en.}, 10(3):82--83, 1976.

\bibitem{LittleFFA2013}
John~B. Little.
\newblock Remarks on generalized toric codes.
\newblock {\em Finite Fields Appl.}, 24:1--14, 2013.

\bibitem{HLRV2014}
Hiram~H. L\'{o}pez and Rafael~H. Villarreal.
\newblock Computing the degree of a lattice ideal of dimension one.
\newblock {\em J. Symbolic Comput.}, 65:15--28, 2014.

\bibitem{MS05-MultHPolynomial}
Diane Maclagan and Gregory~G. Smith.
\newblock Uniform bounds on multigraded regularity.
\newblock {\em J. Algebraic Geom.}, 14(1):137--164, 2005.

\bibitem{JBYPRV2017}
Jos\'{e} Mart\'{i}nez-Bernal, Yuriko Pitones, and Rafael~H. Villarreal.
\newblock Minimum distance functions of graded ideals and {R}eed-{M}uller-type
  codes.
\newblock {\em J. Pure Appl. Algebra}, 221(2):251--275, 2017.

\bibitem{CombComAlgBook}
Ezra Miller and Bernd Sturmfels.
\newblock {\em Combinatorial Commutative Algebra}.
\newblock Cambridge Studies in Advanced Mathematics. Springer-Verlag New York,
  2005.

\bibitem{MT2005}
Marcel Morales and Apostolos Thoma.
\newblock Complete intersection lattice ideals.
\newblock {\em J. Algebra}, 284(2):755--770, 2005.

\bibitem{O'Carroll14}
Liam O'Carroll, Francesc Planas-Vilanova, and Rafael~H. Villarreal.
\newblock Degree and algebraic properties of lattice and matrix ideals.
\newblock {\em SIAM J. Discrete Math.}, 28(1):394--427, 2014.

\bibitem{AlgebraicMethods2011}
Carlos Renter\'{i}a-M\'{a}rquez, Aron Simis, and Rafael~H. Villarreal.
\newblock Algebraic methods for parameterized codes and invariants of vanishing
  ideals over finite fields.
\newblock {\em Finite Fields Appl.}, 17(1):81--104, 2011.

\bibitem{DR2009}
Diego Ruano.
\newblock On the structure of generalized toric codes.
\newblock {\em J. Symbolic Comput.}, 44(5):499--506, 2009.

\bibitem{ToricCIcodesIS13}
Ivan Soprunov.
\newblock Toric complete intersection codes.
\newblock {\em J. Symbolic Comput.}, 50:374--385, 2013.


\bibitem{sahin18}
Mesut \c{S}ahin.
\newblock Toric codes and lattice ideals.
\newblock {\em Finite Fields Appl.}, 52:243--260, 2018.

\bibitem{MultHFuncToricCodes16}
Mesut \c{S}ahin and Ivan Soprunov.
\newblock Multigraded {H}ilbert functions and toric complete intersection
  codes.
\newblock {\em J. Algebra}, 459:446--467, 2016.

\bibitem{VillarrealBOOK}
Rafael~H. Villarreal.
\newblock {\em Monomial algebras}.
\newblock Monographs and Research Notes in Mathematics. CRC Press, Boca Raton,
  FL, second edition, 2015.

\end{thebibliography}

\end{document}